\documentclass[a4paper, 10pt]{amsart}
\usepackage{amsmath}
\usepackage{amssymb, color, hyperref}
\usepackage{amsfonts}

\theoremstyle{plain}
\newtheorem{theorem}{\bf Theorem}[section]
\newtheorem{proposition}[theorem]{\bf Proposition}
\newtheorem{lemma}[theorem]{\bf Lemma}
\newtheorem{corollary}[theorem]{\bf Corollary}

\theoremstyle{definition}

\newtheorem{definition}[theorem]{\bf Definition}
\newtheorem{remark}[theorem]{\bf Remark}

\newcommand{\N}{\mathbb N}
\newcommand{\Z}{\mathbb Z}

\renewcommand{\t}{\, | \,}
\newcommand{\und}{\;\mbox{ and }\;}

\newcommand{\la}{\langle}
\newcommand{\ra}{\rangle}

 \DeclareMathOperator{\ord}{ord}
\DeclareMathOperator{\lcm}{lcm} 

 \DeclareMathOperator{\supp}{supp}

\newcommand{\bdot}{\boldsymbol{\cdot}}
\newcommand{\bulletprod}[1]{\underset{#1}{\bullet}}
\newcommand{\red}{{\text{\rm red}}}

\numberwithin{equation}{section}

\begin{document}

\title[Product-one sequences]{On the Algebraic and Arithmetic structure \\ of the monoid of Product-one sequences II}

\author{Jun Seok Oh}
\address{Institute for Mathematics and Scientific Computing\\ University of Graz, NAWI Graz\\ Heinrichstra{\ss}e 36\\ 8010 Graz, Austria }
\email{junseok.oh@uni-graz.at}

\thanks{This work was supported by the Austrian Science Fund FWF, W1230 Doctoral Program Discrete Mathematics.}


\begin{abstract}
Let $G$ be a finite group and $G'$ its commutator subgroup. By a sequence over $G$, we mean a finite unordered sequence of terms from $G$, where repetition is allowed, and we say that it is a product-one sequence if its terms can be ordered such that their product equals the identity element of $G$. The monoid $\mathcal B (G)$ of all product-one sequences over $G$ is a finitely generated C-monoid whence it has a finite commutative class semigroup. It is well-known that the class semigroup is a group if and only if $G$ is abelian (equivalently, $\mathcal B (G)$ is Krull). In the present paper we show that the class semigroup is Clifford (i.e., a union of groups) if and only if $|G'| \le 2$ if and only if $\mathcal B (G)$ is seminormal, and we study sets of lengths in $\mathcal B (G)$.
\end{abstract}

\maketitle


\section{Introduction} \label{1}

Let $G$ be a finite group and $G'$ its commutator subgroup. A sequence $S$ over $G$ means a finite sequence of terms from $G$ which is unordered and repetition of terms is allowed. We say that $S$ is a product-one sequence if its terms can be ordered such that their product equals the identity element of the group. Clearly juxtaposition of sequences is a commutative operation on the set of sequences. As usual we consider sequences as elements of the free abelian monoid $\mathcal F (G)$ with basis $G$, and clearly the subset $\mathcal B (G) \subset \mathcal F (G)$ of all product-one sequences is a submonoid. The focus on the present paper is on non-abelian finite groups. Sequences over general (not necessarily abelian) finite groups have been studied in combinatorics since the work of Olson (\cite{Ol-Wh77} for an upper bound on the small Davenport constant) and there has been renewed interest (\cite{Ba07b, Ga-Lu08a, Ga-Li10b,Ha15a, Ge-Gr13a, Gr13b, Br-Ri18a}), partly motivated by connections to invariant theory (\cite{ Cz-Do-Ge16a, Cz-Do13c, Cz-Do14a, Cz-Do15a, Cz-Do-Sz17}).

In the present paper we continue the work started in \cite{Oh18a}. The monoid $\mathcal B (G)$ is a finitely generated C-monoid whence, by definition of a C-monoid, its class semigroup is finite. A large class of Mori domains (so-called C-domains, see \cite{Ge-HK06a} for basics, and \cite{Re13a, Ge-Ra-Re15c, Ge-Ha08b, Ka16b}) are known to have finite class semigroup, and this allow one to derive arithmetical finiteness results. However,  the algebraic structure of their class semigroups is unknown and monoids of product-one sequences, being  combinatorial C-monoids,  represent the first class of C-monoids for which we have some first insight into their structure. Among others,   it was shown in \cite{Oh18a} that the class semigroup of $\mathcal B (G)$ is a group if and only if $\mathcal B (G)$ is Krull (resp., root closed) if and only if $G$ is abelian (Theorem \ref{3.1}). In the present paper we provide a characterization of idempotent elements in the class semigroup (Proposition \ref{3.3}) which allows us to show that the class semigroup of $\mathcal B (G)$ is Clifford (i.e., a union of groups) if and only if $\mathcal B (G)$ is seminormal if and only if $|G'| \le 2$ (Theorme \ref{3.6}, \ref{3.11}, and Corollary \ref{3.12}).

Let $H$ be a transfer Krull monoid over a finite abelian group $G$ (this includes commutative Krull monoids with class group $G$ having prime divisors in all classes). Then the arithmetic of $H$ and of $\mathcal B (G)$ are closely connected and their systems of all sets of lengths coincide.  This is the reason why the study of sets of lengths in $\mathcal B (G)$ (for finite abelian groups $G$) is a central topic in factorization theory. In Section \ref{4} we take first steps towards studying sets of lengths in $\mathcal B (G)$ for non-abelian groups. Among others, we show that over the dihedral group and over non-abelian groups with small Davenport constant, sets of lengths are different from sets of lengths over any abelian group (Theorems \ref{4.4} and \ref{4.7}).  At the beginning of Section \ref{4} we provide a detailed discussion of the involved topics.

\bigskip
\section{Background on  product-one sequences and their class semigroups} \label{2}
\bigskip

Our notation and terminology are consistent with the first part \cite{Oh18a} and also with \cite{Ge-Gr13a, Cz-Do-Ge16a}. We briefly gather the key notions. To begin with,  by $\N$ we mean the set of positive integers, and for integers $a, b \in \Z$, $[a,b] = \{x \in \Z \mid a \le x \le b\}$ is the discrete interval.

\smallskip
\noindent
{\bf Groups.} Let $G$ be a multiplicatively written, finite group. For an element $g \in G$, $\ord (g) \in \N$ is its order, and for a subset $G_0 \subset G$, $\langle G_0 \rangle \subset G$ denotes the subgroup generated by $G_0$.  Furthermore,
\begin{itemize}
\item $\mathsf Z (G) = \{g \in G \mid gx = xg \ \text{for all} \ x \in G\} \triangleleft G$ is the {\it center} of $G$,

\item $[x,y] = xyx^{-1}y^{-1} \in G$ is the {\it commutator} of the elements $x,y \in G$, and

\item $G' = [G,G] = \langle [x,y] \mid x,y \in G \rangle \triangleleft G$ is the    {\it commutator subgroup} of $G$.
\end{itemize}
For every $n \in \N$, we denote
\begin{itemize}
\item  by $C_n$ a {\it cyclic group} of order $n$,

\item by $D_{2n} = \{1_G, a, \ldots, a^{n-1}, b, ab, \ldots, a^{n-1}b\} $ a {\it dihedral group} of order $2n$,

\item by $Dic_{4n} = \langle a, b \t a^{2n} = 1_G, b^{2} = a^{n} \und ba = a^{-1}b \rangle$ a {\it dicyclic group} of order $4n$,

\item by $A_n$ an {\it alternating group} of degree $n$, and

\item by $Q_8 = \{ E, I, J, K, -E, -I, -J, -K \}$  the {\it quaternion group}.
\end{itemize}

\smallskip
\noindent
{\bf Semigroups.} All our semigroups are commutative and have an identity element.
Let $S$ be a semigroup. We denote by $S^{\times}$ its group of invertible elements and by $\mathsf E (S)$ the set of all idempotents of $S$, endowed with the {\it Rees order} $\leq$, defined by $e \leq f$ if $ef =e$. Clearly, $ef \leq e$ and $ef \le  f$ for all $e,f \in \mathsf E (S)$. If $E \subset \mathsf E (S)$ is a finite subsemigroup,  then $E$ has a smallest element. If $S$ is finite, then for each $a \in S$, there exists an $n \in \N$ such that $a^{n} \in \mathsf E (S)$. For subsets $A, B \subset S$ and $a \in S$, we set
\[
  AB = \{ ab \t a \in A, b \in B \} \und aB = \{ ab \t b \in B \} \,.
\]

A subset $I \subset S$ is called an {\it ideal} if $SI \subset I$.  If $I \subset S$ is an ideal, we define the {\it Rees quotient} to be the semigroup $S/I = (S \setminus I) \cup \{0\}$, where $0$ is a zero element, the product $ab$ is defined as in $S$ if $a,b$ and $ab$ all belong to $S \setminus I$, and $ab = 0$ otherwise.

By a {\it monoid}, we mean a  semigroup  which satisfies the cancellation laws. Let $H$ be a monoid. Then  $\mathsf q (H)$ denotes the quotient group of $H$ and $\mathcal A (H)$ the set of irreducibles (atoms) of $H$. The monoid $H$ is called {\it atomic} if every non-unit of $H$ can be written as a finite product of atoms. We say that $H$ is reduced if $H^{\times} = \{1\}$, and we denote by $H_{\red} =H/H^{\times} = \{aH^{\times} \mid a \in H \}$ the associated reduced monoid of $H$. A monoid $F$ is called {\it free abelian with basis $P \subset F$} if every $a \in F$ has a unique representation of the form
\[
  a = \prod_{p \in P} p^{\mathsf v_p (a)} \quad \text{with} \quad \mathsf v_p (a)=0 \quad  \text{for almost all} \quad p \in P \,.
\]
If $F$ is free abelian with basis $P$, then $P$ is the set of primes of $F$, we set $F = \mathcal F (P)$, and denote by
\begin{itemize}
\item $|a| = \sum_{p \in P} \mathsf v_p (a)$ the {\it length} of $a$, and by

\item $\supp (a) = \{p \in P \mid \mathsf v_p (a)>0\}$ the {\it support} of $a$.
\end{itemize}
A monoid $F$ is factorial if and only if $F_{\red}$ is free abelian if and only if $F$ is atomic and every atom is a prime. We denote by
\begin{itemize}
\item $H' = \{ x \in \mathsf q (H) \mid \text{there is an $N \in \N$ such that } \ x^n \in H \ \text{for all} \ n \ge N\}$ the {\it seminormalization} of $H$,  by

\item $\widetilde H = \{ x \in \mathsf q (H) \mid x^N \in H \ \text{for some} \ N \in \N\}$ the {\it root closure} of $H$, and by

\item $\widehat H = \{ x \in \mathsf q (H) \mid \text{there is a $c \in H$ such that} \ cx^n \in H \ \text{for all} \ n \in \N \}$ the {\it complete integral closure} of $H$,
\end{itemize}
and observe that $H \subset H' \subset \widetilde H \subset \widehat H \subset \mathsf q (H)$. Then the monoid $H$ is called
\begin{itemize}
\item {\it seminormal} if $H = H'$ (equivalently,  if $x \in \mathsf q (H)$ and $x^2, x^3 \in H$, then $x \in H$),

\item {\it root closed} if $H = \widetilde H$,

\item {\it completely integrally closed} if $H =\widehat H$.
\end{itemize}
A monoid homomorphism $\varphi \colon H \to D$ is said to be
\begin{itemize}
\item a {\it divisor homomorphism} if $a, b \in H$ and $\varphi (a) \mid \varphi (b)$ implies that $a \mid b$.

\item a {\it divisor theory} if $D$ is free abelian, $\varphi$ is a divisor homomorphism, and for all $\alpha \in D$ there are $a_1, \ldots, a_m \in H$ such that $\alpha = \gcd \big( \varphi (a_1), \ldots, \varphi (a_m) \big)$.
\end{itemize}

\noindent
 A monoid $H$ is said to be a {\it Krull monoid} if it satisfies one of the following equivalent conditions (see \cite[Theorem 2.4.8]{Ge-HK06a}){\rm \,:}
\begin{enumerate}
\item[(a)] $H$ is completely integrally closed and satisfies the ACC on divisorial ideals.

\smallskip
\item[(b)] $H$ has a divisor theory.
\end{enumerate}

\smallskip
\noindent
{\bf Class semigroups and C-monoids.} (a detailed presentation can be found in \cite[Chapter 2]{Ge-HK06a}).
Let $F$ be a monoid and $H \subset F$ a submonoid. For any two elements $y, y' \in F$, we define $H$-equivalence $\sim_H$ by
\begin{equation} \label{equi}
y \sim_H y' \quad \textnormal{if} \quad y^{-1}H \cap F = {y'}^{-1} H \cap F \,.
\end{equation}
Then $H$-equivalence is a congruence relation on $F$. For $y \in F$, let $[y]_H^F$ denote the congruence class of $y$, and let
\[
\mathcal C (H,F) = \big\{ [y]_H^F \mid y \in F \big\} \quad \text{and} \quad \mathcal C^* (H,F) = \big\{ [y]_H^F \mid y \in (F \setminus F^{\times}) \cup \{1\} \big\} \,.
\]
Then $\mathcal C (H,F)$ is a commutative semigroup with unit element $[1]_H^F$ (called the {\it class semigroup} of $H$ in $F$) and $\mathcal C^* (H,F) \subset \mathcal C (H,F)$ is a subsemigroup (called the {\it reduced class semigroup} of $H$ in $F$).
As usual, class groups and class semigroups will both be written additively.

A monoid $H$ is called a {\it {\rm C}-monoid} if $H$ is a submonoid of a factorial monoid $F$ such that $H \cap F^{\times} = H^{\times}$ and $\mathcal C^* (H, F)$ is finite.
A Krull monoid is a C-monoid if and only if it has finite class group. We refer to \cite{Ge-HK06a, Ge-Ra-Re15c, Re13a, Ka16b} for more on C-monoids.

\smallskip
\noindent
{\bf Sequences over groups.}
Let $G$ be a finite group with identity element $1_G$ and $G_0 \subset G$ a subset. The elements of the free abelian monoid $\mathcal F (G_0)$ will be called  {\it sequences} over $G_0$.  This terminology goes back to Combinatorial Number Theory. Indeed, a sequence over $G_0$ can be viewed as a finite unordered sequence of terms from $G_0$, where the repetition of elements is allowed. In order to avoid confusion between multiplication in $G$ and multiplication in $\mathcal F (G_0)$, we denote multiplication in $\mathcal F (G_0)$ by the boldsymbol $\bdot$ and we use brackets for all exponentiation in $\mathcal F (G_0)$. In particular, a sequence $S \in \mathcal F (G_0)$ has the form
\begin{equation} \label{basic}
S = g_1 \bdot \ldots \bdot g_{\ell} = \bulletprod{i\in [1,\ell]} g_i = \bulletprod{g \in G_0}g^{[\mathsf v_g (S)]} \in \mathcal F (G_0),
\end{equation}
where $g_1, \ldots, g_{\ell} \in G_0$ are the terms of $S$. Moreover, if $S_1, S_2 \in \mathcal F (G_0)$ and $g_1, g_2 \in G_0$, then $S_1 \bdot S_2 \in \mathcal F (G_0)$ has length $|S_1|+|S_2|$, \ $S_1 \bdot g_1 \in \mathcal F (G_0)$ has length $|S_1|+1$, \ $g_1g_2 \in G$ is an element of $G$, but $g_1 \bdot g_2 \in \mathcal F (G_0)$ is a sequence of length $2$. If $g \in G_0$, $T \in \mathcal F (G_0)$, and $k \in \N_0$, then
\[
g^{[k]}=\underset{k}{\underbrace{g\bdot\ldots\bdot g}}\in \mathcal F (G_0) \quad \text{and} \quad T^{[k]}=\underset{k}{\underbrace{T\bdot\ldots\bdot T}}\in \mathcal F (G_0) \,.
\]
Let $S \in \mathcal F (G_0)$ be a sequence as in \eqref{basic}. Then we denote by
\[
\pi (S) = \{ g_{\tau (1)} \ldots  g_{\tau (\ell)} \in G \mid \tau\mbox{ a permutation of $[1,\ell]$} \} \subset G \quad \und \quad \Pi (S) = \underset{|T| \ge 1}{\bigcup_{T \t S}} {\pi}(T)  \subset G \,,
\]
the {\it set of products} and {\it subsequence products} of $S$, and it can easily be seen that $\pi (S)$ is contained in a $G'$-coset. Note that $|S|=0$ if and only if $S=1_{\mathcal F (G)}$, and in that case we use the convention that $\pi (S) = \{1_G\}$. The sequence $S$ is called
\begin{itemize}
\item a {\it product-one sequence} if $1_G \in \pi (S)$, and

\item {\it product-one free} if $1_G \notin \Pi (S)$.
\end{itemize}

\smallskip
\noindent
If $S = g_1 \bdot \ldots \bdot g_{\ell} \in \mathcal B (G)$ is a product-one sequence with $1_G = g_1 \ldots g_{\ell}$, then $1_G = g_i \ldots g_{\ell}g_1 \ldots g_{i-1}$ for every $i \in [1, \ell]$.
Every map of groups $\theta : G \rightarrow H$ extends to a monoid homomorphism $\theta : \mathcal F (G) \rightarrow \mathcal F (H)$, where $\theta (S) = \theta (g_1)\bdot \ldots \bdot \theta (g_{\ell})$.
If $\theta$ is a group homomorphism, then $\theta (S)$ is a product-one sequence if and only if $\pi(S) \cap \ker(\theta) \neq \emptyset$.

\smallskip
\begin{definition} \label{2.1}~
Let $G_0 \subset G$ be a subset.
\begin{enumerate}
\item The submonoid
      \[
      \mathcal B (G_0) = \{S \in \mathcal F (G_0) \mid 1_G \in \pi (S) \} \subset \mathcal F (G_0)
      \]
      is called the {\it monoid of product-one sequences}, and $\mathcal A (G_0) := \mathcal A \big( \mathcal B (G_0) \big)$ is its set of atoms.

\smallskip
\item We call
      \[
      \mathsf D (G_0) = \sup \{ |S| \mid S \in \mathcal A (G_0) \} \in \N \cup \{\infty\}
      \]
      the {\it large Davenport constant} of $G_0$ and
      \[
      \mathsf d (G_0) = \sup \{ |S| \mid S \in \mathcal F  (G_0) \ \text{is product-one free} \} \in \N_0 \cup \{\infty\}
      \]
      the {\it small Davenport constant} of $G_0$.
\end{enumerate}
\end{definition}

\smallskip
The following elementary lemma will be used without further mention (see \cite[Lemma 3.1]{Cz-Do-Ge16a}).

\smallskip
\begin{lemma} \label{2.2}~
Let  $G_0 \subset G$ be a subset.
\begin{enumerate}
\item $\mathcal B (G_0)$ is a reduced finitely generated C-monoid in $\mathcal F (G)$, $\mathcal A (G_0)$ is finite, and $\mathsf D (G_0) \le |G|$.

\smallskip
\item Let $S \in \mathcal F (G)$ be product-one free.
      \begin{enumerate}
      \smallskip
      \item If $g_0 \in \pi (S)$, then $g_0^{-1} \bdot S \in \mathcal A (G)$. In particular, $\mathsf d (G)+1 \le \mathsf D (G)$.

      \smallskip
      \item If $|S|=\mathsf d (G)$, then $\Pi (S) = G \setminus \{1_G\}$ and hence
            \[
              \mathsf d (G) = \max \big\{ |S| \mid S \in \mathcal F (G) \, \mbox{ with } \, \Pi (S) = G \setminus \{1_G\} \big\} \,.
            \]
      \end{enumerate}

\smallskip
\item If $G$ is cyclic, then $\mathsf d (G)+1 = \mathsf D (G) = |G|$.
\end{enumerate}
\end{lemma}


\section{On the structure of  class semigroups} \label{3}
\bigskip

\centerline{\it Throughout this section, let $G$ be a finite  group with identity  $1_G \in G$ and commutator group $G'$.}
\bigskip

Our starting point is the following characterization (\cite[Theorem 3.2]{Cz-Do-Ge16a} and \cite[Proposition 3.4]{Oh18a}).

\begin{theorem}  \label{3.1}~
The following statements are equivalent{\rm \,:}
\begin{enumerate}
\item[(a)] $G$ is abelian.

\smallskip
\item[(b)] The embedding $\mathcal B (G) \hookrightarrow \mathcal F (G)$ is a divisor homomorphism.

\smallskip
\item[(c)] $\mathcal B (G)$ is root closed.

\smallskip
\item[(d)] $\mathcal B (G)$ is a Krull monoid.

\smallskip
\item[(e)] $\mathcal B (G)$ is a transfer Krull monoid.

\smallskip
\item[(f)] $\mathcal C \big(\mathcal B (G), \mathcal F (G)\big)$ is a group.
\end{enumerate}
If this is the case, then $\mathcal C \big(\mathcal B (G), \mathcal F (G)\big) \cong G$.
\end{theorem}

\smallskip
The main goal of this section is to characterize when the class semigroup $\mathcal C \big(\mathcal B (G), \mathcal F (G)\big)$ is the union of groups.
Let $\mathcal C$ be an additively written semigroup. For $a,b \in \mathcal C$, we define {\it Green's relation} $\mathcal H$ by
\[
  a \mathcal H b  \quad \mbox{ if } \,\, a \in b + \mathcal C \und b \in a + \mathcal C \,\,\, (\mbox{equivalently}, \, a+\mathcal C = b + \mathcal C) \,.
\]
Then $\mathcal H$ is a congruence relation on $\mathcal C$, and for $a \in \mathcal C$, we denote by $\mathcal H (a)$ the congruence class of $a$. The following lemma describes a relation between idempotent elements and maximal subgroups (see \cite[Proposition I.4.3 and Corollary I.4.5]{Gr01}).

\begin{lemma} \label{3.2}~
Let $\mathcal C$ be a semigroup and $\mathcal H$ Green's relation on $\mathcal C$.
\begin{enumerate}
\item An $\mathcal H$-class is a subgroup of $\mathcal C$ if and only if it contains an idempotent element of $\mathcal C$. In particular, if $\mathcal H (e)$ is a group, then $e$ is the identity element and the unique idempotent element. 

\smallskip
\item $\big\{ \mathcal H (e) \t e \in \mathsf E (\mathcal C) \big\}$ is the set of all maximal subgroups of $\mathcal C$, and they are pairwise disjoint.
\end{enumerate}
\end{lemma}

\smallskip
An element $a \in \mathcal C$ is called {\it regular} if $a$ lies in a subgroup of $\mathcal C$. Clearly $a \in \mathcal C$ is regular if and only if there is an $e \in \mathsf E (\mathcal C)$ such that $a \in \mathcal H (e)$. A semigroup $\mathcal C$ is called a {\it Clifford semigroup} if every element of $\mathcal C$ is regular. Thus $\mathcal C$ is a Clifford semigroup if and only if $\mathcal C$ is the union of $\mathcal H (e)$ for all $e \in \mathsf E (\mathcal C)$.

Applying (\ref{equi}) to the inclusion $\mathcal B (G) \subset \mathcal F (G)$, we say that two sequences $S, S' \in \mathcal F (G)$ are $\mathcal B (G)$-equivalent if one of the following  equivalent conditions hold{\rm \,:}
\begin{itemize}
\item For all $T \in \mathcal F (G)$, we have $S \bdot T \in \mathcal B (G)$ if and only if $S' \bdot T \in \mathcal B (G)$.

\item For all $T \in \mathcal F (G)$, we have $1_G \in \pi ( S \bdot T)$ if and only if $1_G \in \pi ( S' \bdot T)$.
\end{itemize}
We write $\sim$ instead of $\sim_{\mathcal B (G)}$ and set $[S] = [S]_{\mathcal B (G)}^{\mathcal F (G)}$. If $S, S' \in \mathcal F (G)$ are sequences with $S \sim S'$, then $\pi(S) = \pi(S')$ (see \cite[Lemma 3.6]{Oh18a}). We will use this simple fact without further mention.

\smallskip
\begin{proposition} \label{3.3}~
Let $H \subset G'$ be a subset. Then the following statements are equivalent{\rm \,:}
\begin{enumerate}
\item[(a)] $H = \pi(S)$ for some sequence $S \in \mathcal F (G)$ with $[S] \in \mathsf E \Big( \mathcal C \big(\mathcal B (G), \mathcal F (G)\big) \Big)$.

\smallskip
\item[(b)] $H = G'_0$ for a subgroup $G_0 \subset G$.
\end{enumerate}
If $(a)$ holds, then $\la \supp(S) \ra$ is the desired subgroup.
\end{proposition}

\begin{proof}
(a) $\Rightarrow$ (b) Let $S = g_1 \bdot \ldots \bdot g_{\ell} \in \mathcal F (G)$ be such that $\pi(S) = H$ and $[S]$ is idempotent. Then $\pi(S) = \pi(S^{[n]})$ for all $n \in \N$, and $H = \pi(S) \subset G'$ is a subgroup by \cite[Lemma 3.7.1]{Oh18a}. We set $G_0 = \langle g_1, \ldots, g_{\ell} \rangle$ and assert that $\pi(S) = G'_0$. Clearly we have $\pi(S) \subset G'_0$.

Conversely, we first show that $[g_i, g_j] \in \pi(S)$ for all $i, j \in [1,\ell]$. Let $i,j \in [1,\ell]$ and consider the commutator $[g_i, g_j] = g_{i}g_{j}g^{-1}_{i}g^{-1}_{j}$.
We may assume that $i \neq j$, otherwise $[g_i, g_j] = 1_G \in \pi(S)$.
Then we can take product-one equations, in $\pi(S)$, with $g_i$ and $g_j$ being the first elements, say
\[
  g_{i}g_{n_1}\ldots g_{n_{i-1}}g_{n_{i+1}}\ldots g_{n_{\ell}} = 1_G \quad \und \quad g_{j}g_{m_1}\ldots g_{m_{j-1}}g_{m_{j+1}}\ldots g_{m_{\ell}} = 1_G \,.
\]
It follows that
\[
 g^{-1}_i = g_{n_1}\ldots g_{n_{i-1}}g_{n_{i+1}}\ldots g_{n_{\ell}} \quad \und \quad g^{-1}_j = g_{m_1}\ldots g_{m_{j-1}}g_{m_{j+1}}\ldots g_{m_{\ell}} \,,
\]
and hence $[g_i, g_j] = g_{i}g_{j}(g_{n_1}\ldots g_{n_{i-1}}g_{n_{i+1}}\ldots g_{n_{\ell}})(g_{m_1}\ldots g_{m_{j-1}}g_{m_{j+1}}\ldots g_{m_{\ell}}) \in \pi(S^{[2]}) = \pi(S)$.
Now let $g, h \in G_0$ and consider its commutator $[g, h] = ghg^{-1}h^{-1}$.
Since $G_0 = \langle g_1, \ldots, g_{\ell} \rangle$, $g$ and $h$ have the form
\[
  g = g^{n_{i_1}}_{i_1}\ldots g^{n_{i_t}}_{i_t} \quad \und \quad h = g^{m_{j_1}}_{j_1}\ldots g^{m_{j_k}}_{j_k} \,,
\]
where $t, k \in \N$, $n_{i_1}, \ldots, n_{i_t}, m_{j_1}, \ldots, m_{j_k} \in \N_0$, and $g_{i_1}, \ldots, g_{i_t}, g_{j_1}, \ldots, g_{j_k} \in \supp(S)$.
For each generator of $G_0$, we can express its inverse as a product of all other generators of $G_0$ from the product-one equation in $\pi(S)$.
Let $M = n_{i_1} + \ldots + n_{i_t} + m_{j_1} + \ldots + m_{j_k} \in \N$.
As at the start of the proof, we obtain that
\[
  [g, h] = g^{n_{i_1}}_{i_1}\ldots g^{n_{i_t}}_{i_t}g^{m_{j_1}}_{j_1}\ldots g^{m_{j_k}}_{j_k}\big((g^{-1}_{i_t})^{n_{i_t}}\ldots (g^{-1}_{i_1})^{n_{i_1}}\big) \big((g^{-1}_{j_k})^{m_{j_k}}\ldots (g^{-1}_{j_1})^{m_{j_1}}\big) \in \pi\big(S^{[M]}\big) = \pi(S) \,.
\]
It follows that all generators of $G'_0$ belong to $\pi(S)$, whence $G'_0 \subset \pi(S)$.

\smallskip
(b) $\Rightarrow$ (a) Let  $G_0 \subset G$ be a subgroup with $G'_0 = H$.
By \cite[Lemma 3.7.3]{Oh18a}, there exists a sequence $S \in \mathcal F (G_0)$ such that $[S]^{\mathcal F (G_0)}_{\mathcal B (G_0)} \in \mathsf E \Big( \mathcal C \big(\mathcal B (G_0), \mathcal F (G_0)\big) \Big)$ and $\pi(S) = G'_0$.
It follows that $\pi(S) = \pi\big(S^{[n]}\big)$ for all $n \in \N$, and since $\mathcal C \big(\mathcal B (G), \mathcal F (G)\big)$ is finite by Lemma \ref{2.2}.1, there is an $m \in \N$ such that $\big[S^{[m]}\big] \in \mathsf E \Big( \mathcal C \big(\mathcal B (G), \mathcal F (G)\big) \Big)$. Thus $\big[S^{[m]}\big]$ has the required property.
\end{proof}

\smallskip
\begin{lemma} \label{3.4}~
Let $S \in \mathcal F (G)$ be a sequence with $[S] \in \mathsf E \Big( \mathcal C \big(\mathcal B (G), \mathcal F (G)\big) \Big)$. Then there is a $S_0 \in \mathcal B (G)$ such that $[S_0] = [S]$ and
\[
  \langle \supp(S') \rangle \subset \langle \supp(S_0) \rangle \quad \mbox{ for all } \,\, S' \in \mathcal B (G) \,\, \mbox{ with } \,\, [S'] = [S] \,.
\]
\end{lemma}

\begin{proof}
If $m = \lcm \{ \ord (g) \t \in \supp (S) \}$, then $S^{[m]} \in \mathcal B (G)$ and $\big[S^{[m]}\big] = [S]$. We set
\[
  X = \bigcup \big\{ \supp(S') \t S' \in \mathcal B (G), [S'] = [S] \big\} \subset G \,,
\]
say $X = \{ g_1, \ldots, g_{\ell} \}$ with $\ell \in \N$. Then for each $i \in [1,\ell]$, we set
\[
  S_i \in \mathcal B (G) \quad \mbox{ such that } \quad [S_i] = [S] \und g_i \in \supp(S_i) \,,
\]
and we define $S_0 = S_1 \bdot \ldots \bdot S_{\ell} \in \mathcal B (G)$.
Since $[S]$ is idempotent element, we have $[S_0] = [S]$, and if $S' \in \mathcal B (G)$ such that $[S'] = [S]$, then $\supp(S') \subset X = \supp(S_0)$. Thus $S_0 = S_1\bdot \ldots \bdot S_{\ell}$ has the required properties.
\end{proof}

\smallskip
\begin{proposition} \label{3.5}~
Let $S \in \mathcal F (G)$ be a sequence with $[S] \in \mathsf E \Big( \mathcal C \big(\mathcal B (G), \mathcal F (G)\big) \Big)$. Then there exist a subgroup $G_0 \subset G$ and a homomorphism  $\varphi_{[S]} \colon \mathcal H \big([S]\big) \to G_0/G'_0$ such that   $\pi(T)$ is a $G_0'$-coset for every sequence $T \in \mathcal F (G)$ with $[T] \in \mathcal H \big([S]\big)$. In particular, we have the following special cases{\rm \,:}
\begin{enumerate}
\item If $[S]$ is the smallest idempotent, then $G_0 = G$ and $\varphi_{[S]}$ is an isomorphism.

\smallskip
\item If $[S]$ is the greatest idempotent, then $G_0 = \mathsf Z (G)$ and $\varphi_{[S]}$ is an isomorphism.
\end{enumerate}
\end{proposition}

\begin{proof}
We may suppose that $S$ satisfies the conditions given in Lemma \ref{3.4}.
Then, by Proposition \ref{3.3}, we obtain that $G_0 = \langle \supp(S) \rangle \subset G$ is a subgroup with $G'_0 = \pi(S)$.

Let $T \in \mathcal F (G)$ be a sequence with $[T] \in \mathcal H \big([S]\big)$. Since $\mathcal H \big([S]\big)$ is finite group by Lemma \ref{3.2}.1 and Lemma \ref{2.2}.1, it follows that there is an $n \in \N$ such that $\big[T^{[n]}\big] = [S]$. From the choice of $S$, we have that
\[
  \langle \supp(T) \rangle = \langle \supp(T^{[n]}) \rangle \subset \langle \supp(S) \rangle = G_0 \,.
\]
It follows that $T \in \mathcal F (G_0)$ and hence $\pi(T)$ is contained in a $G'_0$-coset. Since $[T] \in \mathcal H \big([S]\big)$ and $1_G \in \pi(S)$,
\[
 [T] = [T\bdot S] \quad \mbox { implies that } \quad \pi(T) \subset \pi(T)\pi(S) \subset \pi(T\bdot S) = \pi(T) \,,
\]
whence $\pi(T) = \pi(T)\pi(S)$. Thus we obtain that $\pi(T) = gG'_0$ for some $g \in \pi(T) \subset G_0$, and therefore the map
\[
    \begin{aligned}
      \varphi_{[S]} : \mathcal H \big([S]\big) \quad & \to \quad G_0/G'_0 \\
                                    [T]        \quad & \mapsto  \quad \pi(T) = gG'_0
    \end{aligned}
\]
is the desired homomorphism. Indeed, if $[T], [T'] \in \mathcal H \big([S]\big)$, then $\pi(T\bdot T') = (gg')G'_0$, where $g \in \pi(T)$ and $g' \in \pi(T')$, whence
\[
  \varphi_{[S]} \big([T] + [T'] \big) = \varphi_{[S]} \big([T\bdot T']\big) = (gg')G'_0 = (gG'_0)(g'G'_0) = \varphi_{[S]}\big([T]\big)\varphi_{[S]}\big([T']\big) \,.
\]

\smallskip
1. Let $[S]$ be the smallest idempotent. We show that $\langle \supp (S) \rangle = G$. If $g \in G$, then $g\bdot S \in \mathcal F (G)$ with $\pi(g\bdot S) = gG'$ by \cite[Lemma 3.7.3]{Oh18a}.
Then there is an $n \in \N$ such that $\pi(g^{[n]} \bdot S) = G'$. It follows that $g^{[n]} \bdot S \sim S$, and thus $\langle \supp(g \bdot S) \rangle = \langle \supp(g^{[n]} \bdot S) \rangle \subset \langle \supp(S) \rangle$, where the last inclusion follows from the maximality property in Lemma \ref{3.4}. Therefore we obtain that $g \in \langle \supp(S) \rangle$ whence $\langle \supp(S) \rangle = G$. Thus Lemma \ref{3.2} and \cite[Theorem 3.8.1]{Oh18a} imply that $\mathcal H \big([S]\big)$ is isomorphic to $G/G'$.

\smallskip
2. Since $[1_{\mathcal F (G)}]$ is the greatest idempotent,  we have $[S] = [1_{\mathcal F (G)}]$. We show that $\langle \supp (S) \rangle = \mathsf Z (G)$. Since $[1_{\mathcal F (G)}] = \mathcal B \big(\mathsf Z (G)\big)$ by \cite[Lemma 3.7.2]{Oh18a}, it follows that $\langle \supp(S) \rangle \subset \mathsf Z (G)$.
If $g \in \mathsf Z (G)$, then $g\bdot g^{-1} \in \mathcal B \big(\mathsf Z (G)\big)$. It follows that $g\bdot g^{-1} \sim S$, and hence the maximality property in Lemma \ref{3.4} implies $\langle g, g^{-1} \rangle \subset \langle \supp(S) \rangle$. Therefore we obtain that $g \in \langle \supp(S) \rangle$ whence $\langle \supp(S) \rangle = \mathsf Z (G)$. Thus Lemma \ref{3.2} and \cite[Theorem 3.8.2]{Oh18a} imply that $\mathcal H \big([1_{\mathcal F (G)}]\big)$ is isomorphic to $\mathsf Z (G)$.
\end{proof}

\smallskip
Let $\mathcal C = \mathcal C \big(\mathcal B (G), \mathcal F (G)\big)$, and let $S \in \mathcal F (G)$ be a sequence with $[S] \in \mathsf E (\mathcal C)$. Since $[S]$ is idempotent, it follows that $[S] + \mathcal C$ is a subsemigroup of $\mathcal C$ with identity element $[S]$. Now we state the first part of our main result, namely that $\mathcal B (G)$ is seminormal if and only if its class semigroup is Clifford.

\smallskip
\begin{theorem} \label{3.6}~
Let $\mathcal C = \mathcal C \big(\mathcal B (G), \mathcal F (G)\big)$. Then the following statements are equivalent{\rm \,:}
\begin{enumerate}
\item[(a)] $\mathcal C$ is a Clifford semigroup.

\smallskip
\item[(b)] $\mathcal B (G)$ is a seminormal monoid.
\end{enumerate}
Suppose that $\mathcal C$ is a Clifford semigroup. Let $S_1, \ldots, S_n \in \mathcal F (G)$ be such that $\mathsf E (\mathcal C) = \big\{ [S_1], \ldots, [S_n] \big\}$, and let $\mathcal C_i = [S_i] + \mathcal C$ for all $i \in [1,n]$. Then the map
\[
  \varphi \colon \mathcal C \,\, \to \,\, \mathcal C^{*} = \Big\{ \big( [T\bdot S_1], \ldots, [T\bdot S_n] \big) \t T \in \mathcal F (G) \Big\} \,\, \subset \,\, \prod_{i=1}^{n} \mathcal C_i \,,
\]
defined by $\varphi \big( [T] \big) = \big( [T\bdot S_1], \ldots, [T\bdot S_n] \big)$ for all $T \in \mathcal F (G)$, is an isomorphism with the following properties{\rm \,:}
\begin{enumerate}
\item If $i, j \in [1,n]$ and $T \in \mathcal F (G)$ is a sequence with $[T] \in \mathcal H \big([S_i]\big)$, then
      \[
         [S_j] \,\, \le \,\, [S_i] \quad \mbox{ if and only if } \quad [T\bdot S_j] \in \mathcal H \big([S_j]\big) \,.
      \]
      In particular, $\mathcal C^{*}$ has no zero element.

\smallskip
\item For every $i \in [1,n]$,
      \[
        \mathcal C_i = \bigcup \mathcal H \big([S_j]\big) \,,
      \]
      where the union is taken over those $j \in [1,n]$ with $[S_j] \le [S_i]$.

\smallskip
\item For every $i \in [1,n]$, the restriction of the projection $p_i{\big|}_{\mathcal C^{*}} \colon \mathcal C^{*} \to \mathcal C_i \big/ (\mathcal C_i \setminus \mathcal C^{\times}_i)$ is surjective.

\smallskip
\item $\mathcal B (G) = \bigcup_{i=1}^{n} \{ T \in \mathcal F (G) \t [T] \in \ker ( \varphi_{[S_i]} ) \}$, where $\varphi_{[S_i]}$ is the homomorphism defined in \textnormal{Prop. \ref{3.5}}.
\end{enumerate}
\end{theorem}

\begin{proof}
(a) $\Rightarrow$ (b)  Let $T \in \mathsf q (\mathcal B (G)) \subset \mathsf q (\mathcal F (G))$ be such that $T^{[2]}, T^{[3]} \in \mathcal B (G)$. Since $\mathcal F (G)$ is free abelian, it is seminormal whence  $T \in \mathcal F (G)$. We have to show that $1_G \in \pi(T)$.
Since $\mathcal C$ is Clifford, there is a $S \in \mathcal F (G)$ such that $[S] \in \mathsf E ( \mathcal C )$ and $[T] \in \mathcal H \big([S]\big)$. Then, by Lemma \ref{3.4} and Proposition \ref{3.5}, we obtain that $\pi(T) = gG'_0$, where $G_0 = \la \supp(S) \ra$ with $\pi(S) = G'_0$ and $g \in G_0$.
Since $T^{[2]}, T^{[3]} \in \mathcal B (G)$, we have that $g^{2}, g^{3} \in G'_0$, and it follows that $g \in G'_0$.
Thus $1_G \in G'_0 = \pi(T)$.

\smallskip
(b) $\Rightarrow$ (a) Let $T \in \mathcal F (G)$ be a sequence. Since $\mathcal C$ is finite by Lemma \ref{2.1}.1, there is an $n \in \N$ such that $\big[ T^{[n]} \big] \in \mathsf E (\mathcal C)$.
It suffices to show that $T^{[n+1]} \sim T$. Indeed, if this holds true, then $[T]$ generates a cyclic subgroup and hence $\mathcal C$ is Clifford.
Let $Z \in \mathcal F (G)$ be any sequence. If $Z \bdot T \in \mathcal B (G)$, then $Z\bdot T^{[n+1]} \in \mathcal B (G)$ because $\big[T^{[n]}\big] \in \mathsf E (\mathcal C)$ implies $T^{[n]} \in \mathcal B (G)$ by Proposition \ref{3.3}.

Conversely, we suppose that $Z\bdot T^{[n+1]} \in \mathcal B (G)$ and have to verify that $Z\bdot T \in \mathcal B (G)$. We assert that
\begin{equation} \label{important}
(Z\bdot T)^{[m]} \bdot T^{[n]} \in \mathcal B (G) \quad \mbox{ for all } \,\, m \in \N_0 \,,
\end{equation}
and we proceed by induction on $m$. It holds for $m=0$ by Proposition \ref{3.3}. If it holds for $m\in \N_0$, then we infer by the induction hypothesis with the equivalence $T^{[n]} \bdot T^{[n]} \sim T^{[n]}$ that
\[
(Z\bdot T)^{[m+1]} \bdot T^{[n]} \sim (Z \bdot T)^{[m]} \bdot T^{[n]} \bdot Z \bdot T^{[n+1]} \in \mathcal B (G)
\]
whence $(Z\bdot T)^{[m+1]} \bdot T^{[n]} \in \mathcal B (G)$. Using Equation \eqref{important} with $m=n$ we infer that
\[
(Z \bdot T)^{[n]} \sim (Z \bdot T)^{[n]} \bdot T^{[n]} \in \mathcal B (G)
\]
whence $(Z \bdot T)^{[n]} \in \mathcal B (G)$. Furthermore,
\[
(Z \bdot T)^{[n+1]} \sim (Z \bdot T)^{[n]} \bdot Z \bdot T^{[n+1]} \in \mathcal B (G)
\]
whence $(Z \bdot T)^{[n+1]} \in \mathcal B (G)$.
Thus there exists an $N \in \N$ such that any integer $\ell \ge N$ can be written as a non-negetive linear combination of the integers $n$ and $n+1$. It follows that $(Z\bdot T)^{[\ell]} \in \mathcal B (G)$ for all $\ell \ge N$. Since $\mathcal B (G)$ is seminormal, we obtain that $Z\bdot T \in \mathcal B (G)$.

\smallskip
Suppose now that $\mathcal C$ is a Clifford semigroup. First we show the statements 1. - 4., and then we prove that $\varphi$ is an isomorphism.

\smallskip
1. Let $i, j \in [1,n]$, and let $T \in \mathcal F (G)$ be a sequence with $[T] \in \mathcal H \big([S_i]\big)$. Since $\mathcal H \big([S_i]\big)$ is a group by Lemma \ref{3.2}.1, it follows that there is a sequence $T' \in \mathcal F (G)$ such that $[T']$ is the inverse of $[T]$ in $\mathcal H \big([S_i]\big)$.
If $[S_j] \le [S_i]$, then
\[
  [T\bdot S_j] = [T\bdot S_j] + [S_j] \und [T\bdot S_j] + [T'\bdot S_j] = [S_i \bdot S_j] = [S_j] \,,
\]
and it follows that $[T\bdot S_j] \in \mathcal H \big([S_j]\big)$. For the converse, we assume $[T\bdot S_j] \in \mathcal H \big([S_j]\big)$. Since $[T] \in \mathcal H \big([S_i]\big)$, we have $[T] = [T \bdot S_i]$. Since $[S_j] \in \mathsf E (\mathcal C)$, it follows that
\[
  [T\bdot S_j] = [T\bdot S_j] + [S_i \bdot S_j] \und [T\bdot S_j] + [T' \bdot S_i \bdot S_j] = [S_i \bdot S_j] \,,
\]
whence $[T\bdot S_j] \in \mathcal H \big([S_j]\big) \cap \mathcal H \big([S_i \bdot S_j]\big)$. Thus Lemma \ref{3.2}.2 implies $[S_i \bdot S_j] = [S_j]$, equivalently $[S_j] \le [S_i]$.

In particular, if $[S_n] \le [S_i]$ for all $i \in [1,n]$, then $\mathcal C_n = \mathcal C^{\times}_n = \mathcal H \big([S_n]\big)$ by \cite[Theorem 3.8.1]{Oh18a}, and it follows that $\mathcal C^{*}$ has no zero element.

\smallskip
2. Observe that, for every $i \in [1,n]$, the map $\mathcal C \to \mathcal C_i$, given by $[T] \mapsto [T\bdot S_i]$ for $T \in \mathcal F (G)$, is an epimorphism, and thus $\mathcal C_i$ is Clifford because the epimomorphic image of Clifford semigroups are Clifford.

Let $i \in [1,n]$ be given, and let $T \in \mathcal F (G)$ be a sequence with $[T] \in \mathcal C_i$. Since $\mathcal C_i$ is Clifford, there is a subgroup $G_0$ of $\mathcal C_i$ containing $[T]$. If we denote by $[S_0]$ the identity element of $G_0$, then $[S_0]$ is idempotent in $\mathcal C$. Since $[S_0] \in \mathcal C_i$, we have that $[S_0] = [Z \bdot S_i]$ for some sequence $Z \in \mathcal F (G)$. Since $[S_i] \in \mathsf E (\mathcal C)$, it follows that
\[
  [S_0 \bdot S_i] = [Z \bdot S_i \bdot S_i] = [Z \bdot S_i] = [S_0] \,,
\]
whence $[S_0] = [S_j]$ for some $j \in [1,n]$ with $[S_j] \le [S_i]$. Thus $[T] \in \mathcal H \big([S_j]\big)$ because $\mathcal C$ is Clifford, and hence $\mathcal C_i \subset \bigcup \mathcal H \big([S_j]\big)$, where the union is taken over those $j \in [1,n]$ with $[S_j] \le [S_i]$. Conversely, if $j \in [1,n]$ with $[S_j] \le [S_i]$ and $T \in \mathcal F (G)$ such that $[T] \in \mathcal H \big([S_j]\big)$, then
\[
  [T] = [T\bdot S_j] = [T\bdot (S_j \bdot S_i)] = [(T\bdot S_j) \bdot S_i] = [T \bdot S_i] \in \mathcal C_i \,\, \mbox{ whence } \,\, \mathcal H \big([S_j]\big) \subset \mathcal C_i \,.
\]

\smallskip
3. It follows from the very definition of the Rees quotient of $\mathcal C_i$ by its ideal $\mathcal C_i \setminus \mathcal C^{\times}_i$ for $i \in [1,n]$.

\smallskip
4. Let $T \in \mathcal B (G)$ be a product-one sequence. Since $\mathcal C$ is Clifford, there exists an $i \in [1,n]$ such that $[T] \in \mathcal H \big([S_i]\big)$. By Proposition \ref{3.5}, $\pi (T)$ is a $G_0'$-coset with $G_0'= \pi (S_i)$. Since $1_G \in \pi (T)$, it follows that $\pi (T) = G_0'$ whence $[T] \in \ker (\varphi_{[S_i]})$. The reverse inclusion runs along the same lines.

\smallskip
Having established 1. - 4. we show that the map 
\[
  \begin{aligned}
   \varphi \colon \mathcal C  & \quad \to \quad \mathcal C^{*} \\
                         [T]  & \quad \mapsto \quad \big( [T\bdot S_1], \ldots, [T\bdot S_n] \big)
  \end{aligned}
\]
is an isomorphism. Observe that, for each sequence $T \in \mathcal F (G)$,
\[
  \varphi \big( [T] \big) = \big( [T\bdot S_1], \ldots, [T\bdot S_n] \big) = \big( \varphi_1 \big([T]\big), \ldots, \varphi_n \big([T]\big) \big) \,,
\]
where for each $i \in [1,n]$, $\varphi_i \colon \mathcal C \to \mathcal C_i$ is an epimorphism.
Then it suffices to show that $\varphi$ is injective. 

To do so, we assume that $\varphi \big([T]\big) = \varphi \big([W]\big)$ for $T, W \in \mathcal F (G)$.
Since $\mathcal C$ is Clifford, there exist $i, j \in [1,n]$ such that $[T] \in \mathcal H \big([S_i]\big)$ and $[W] \in \mathcal H \big([S_j]\big)$.
It suffices to show that $[S_i] = [S_j]$. Indeed, if this holds true, then we obtain that
\[
  [T] = [T \bdot S_i] = [W \bdot S_i] = [W \bdot S_j] = [W] \,.
\]

\smallskip
\noindent CASE 1{\rm \,:} \, $[S_i] = [1_{\mathcal F (G)}]$ or $[S_j] = [1_{\mathcal F (G)}]$.

Without loss of generality, we may assume that $[S_i] = [1_{\mathcal F (G)}]$. Then we have $[T] = [T] + [S_i] = [W] + [S_i] = [W]$, whence $[W] = [T] \in \mathcal H \big([S_i]\big)$. Thus $[S_j] = [S_i]$ by Lemma \ref{3.2}.2.

\smallskip
\noindent CASE 2{\rm \,:} \, $[S_i] \neq [1_{\mathcal F (G)}]$ and $[S_j] \neq [1_{\mathcal F (G)}]$.

Consider the idempotent element $[S_i\bdot S_j]$. We assert that $[S_i \bdot S_j] = [S_i]$. Assume to the contrary that $[S_i \bdot S_j] \neq [S_i]$.
Then, since $[S_i \bdot S_j] \le [S_j]$, Item 1. implies $[W \bdot S_i \bdot S_j] \in \mathcal H \big([S_i \bdot S_j]\big)$. However, since $[S_i \bdot S_j] \neq [S_i]$, it follows that $[W \bdot S_i \bdot S_j] \notin \mathcal H \big([S_i]\big)$ by Lemma \ref{3.2}.2.
Hence $[T] = [T \bdot S_i] = [W \bdot S_i] = [W \bdot S_j \bdot S_i] \notin \mathcal H \big([S_i]\big)$, a contradiction. Thus $[S_i \bdot S_j] = [S_i]$.
By symmetry, we obtain that $[S_i \bdot S_j] = [S_j]$ whence $[S_i] = [S_j]$.
\end{proof}

\smallskip
\begin{remark} \label{3.7}~

Every finite commutative semigroup has a Ponizovsky decomposition (see \cite[Chapter IV.4]{Gr01}). We outline here that the isomorphic image $\mathcal C^{*}$ given in Theorem \ref{3.6} is an explicit description of the Ponizovsky decomposition of $\mathcal C$. Let all notation and assumption be as in Theorem \ref{3.6}. 
\begin{enumerate}
\item Let $i \in [1,n]$. Since every element $[T] \in \mathcal H \big([S_i]\big)$ has the form $[T] = [T] + [S_i]$ and $\mathcal H \big([S_i]\big)$ is a group with identity $[S_i]$, it follows that $\mathcal H \big([S_i]\big) \subset \mathcal C_i^{\times}$. Since $\mathcal C_i \subset \mathcal C$ is a subsemigroup and $\mathcal H \big([S_i]\big)$ is a maximal subgroup of $\mathcal C$, we obtain that $\mathcal H \big([S_i]\big) = \mathcal C_i^{\times}$.

\smallskip
\item By item $2$ in Theorem \ref{3.6}, it is easy to show that, for each $i \in [1,n]$, $\mathcal C_i \setminus \mathcal C_i^{\times} = \bigcup \mathcal C_j$, where the union is taken over those $j \in [1,n]$ with $[S_j] < [S_i]$. By the definition of the Ponizovsky factor in \cite{Gr01}, one can obtain that for each $i \in [1,n]$
\[
  P_{[S_i]} = \mathcal C_i \big/ (\mathcal C_i \setminus \mathcal C_i^{\times}) = \mathcal H \big([S_i]\big) \cup \{ 0_{[S_i]} \} \,,
\]
where $0_{[S_i]}$ is the zero element in Rees quotient. Combining \cite[Proposition IV.4.8]{Gr01} and item $1$ in Theorem \ref{3.6}, we can see that $\mathcal C^{*}$ is the Ponizovsky decomposition.
\end{enumerate}
\end{remark}

\smallskip
\begin{corollary} \label{3.8}~
Let $N  \subset G$ be a normal subgroup.
\begin{enumerate}
\item There is an epimorphism  $\overline \theta \colon \mathcal C \big(\mathcal B (G), \mathcal F (G)\big) \to \mathcal C \big(\mathcal B (G/N), \mathcal F (G/N)\big)$.
    
\smallskip
\item If $\mathcal B (G)$ is seminormal, then $\mathcal B (N)$ and $\mathcal B (G/N)$ are seminormal and their class semigroups are Clifford.
\end{enumerate}        
\end{corollary}

\begin{proof}
1. Let $\theta \colon \mathcal F (G) \to \mathcal F (G/N)$ be the homomorphism extended from the canonical epimorphism $G \to G/N$. We define
\[
  \begin{aligned}
  \bar{\theta} : \mathcal C \big(\mathcal B (G), \mathcal F (G)\big) \quad & \to \quad \mathcal C \big(\mathcal B (G/N), \mathcal F (G/N)\big)\\
                                                                [S]  \quad & \mapsto \quad \big[\theta (S)\big]
  \end{aligned}
\]
In order to verify that $\bar{\theta}$ is well-defined, let $S, S' \in \mathcal F (G)$ be sequences with $S \sim_{\mathcal B (G)} S'$. We have to
show that
\[
  \theta (S) \sim_{\mathcal B (G/N)} \theta (S') \,.
\]
Let $T \in \mathcal F (G/N)$ be a sequence with $T\bdot \theta (S) \in \mathcal B (G/N)$.
Since $\theta$ is surjective, there is a $T_1 \in \mathcal F (G)$ such that $\theta (T_1) = T$, and hence $\theta (T_1\bdot S) = T\bdot \theta (S) \in \mathcal B (G/N)$.
It follows that
\[
  \pi(T_1\bdot S) \cap N \neq \emptyset \,, \und \mbox{say} \,\, h \in \pi(T_1\bdot S) \cap N \,.
\]
Since $S \sim_{\mathcal B (G)} S'$,
\[
  1_G \in \pi(h^{-1}\bdot T_1 \bdot S) \,\, \mbox{ implies that } \,\, 1_G \in \pi(h^{-1}\bdot T_1 \bdot S') \,,
\]
whence $\pi(T_1 \bdot S') \cap N \neq \emptyset$. It follows that $T\bdot \theta (S') = \theta (T_1 \bdot S') \in \mathcal B (G/N)$.

Thus $\bar{\theta}$ is well-defined. Clearly it is surjective and for any sequences $S, S' \in \mathcal F (G)$ we have
\[
  \bar{\theta} \big( [S] + [S'] \big) = \bar{\theta} \big( [S\bdot S'] \big) = \big[\theta (S\bdot S')\big] = \big[\theta (S)\big] + \big[\theta (S')\big] = \bar{\theta} \big([S]\big) + \bar{\theta} \big([S']\big) \,.     
\]

\smallskip
2. Suppose that $\mathcal B (G)$ is seminormal. Since the inclusion $\mathcal B (N) \hookrightarrow \mathcal B (G)$ is a divisor homomorphism, $\mathcal B (N)$ is seminormal by \cite[Lemma 3.2.4]{Ge-Ka-Re15a}. Since epimorphic images of Clifford semigroups are Clifford, the remaining statements follows from 1. and from Theorem \ref{3.6}.
\end{proof}

\smallskip
In our next result we study product-one sequences over the direct product of $G$ with a finite abelian group.  Let $H$ be a finite abelian group.
For any $S = (g_1, h_1)\bdot \ldots \bdot (g_{\ell}, h_{\ell}) \in \mathcal F (G \times H)$, we have
\[
  \pi(S) = \big\{ (g, h) \bigm\vert g \in \pi(g_1 \bdot \ldots \bdot g_{\ell}) \und \{ h \} = \pi(h_1 \bdot \ldots \bdot h_{\ell}) \big\} \,.
\]
In particular,  $S \in \mathcal B (G\times H)$ if and only if $S' = (g_1, h_{\sigma(1)})\bdot \ldots \bdot (g_{\ell}, h_{\sigma(\ell)}) \in \mathcal B (G\times H)$ for any permutation $\sigma$ on $[1,\ell]$.

\smallskip
\begin{theorem} \label{3.9}~
Let $H$ be an abelian group.
\begin{enumerate}
\item The map
\[
  \begin{aligned}
  \mathcal C \big(\mathcal B (G\times H), \mathcal F (G\times H)\big) & \quad \to \quad \mathcal C \big(\mathcal B (G), \mathcal F (G)\big) \times \mathcal C \big(\mathcal B (H), \mathcal F (H)\big)\\
  \big[(g_1, h_1)\bdot \ldots \bdot (g_{\ell}, h_{\ell})\big] & \quad \mapsto \quad \big( [g_1\bdot \ldots \bdot g_{\ell}], [h_1 \bdot \ldots \bdot h_{\ell}]\big)
  \end{aligned}
\]
is a semigroup isomorphism.

\item The following statements are equivalent{\rm \,:}
\begin{enumerate}
\smallskip
\item[(a)] $\mathcal B (G)$ is a seminormal monoid.

\smallskip
\item[(b)] $\mathcal C \big(\mathcal B (G), \mathcal F (G)\big)$ is a Clifford semigroup.

\smallskip
\item[(c)] $\mathcal C \big(\mathcal B (G \times H), \mathcal F (G \times H)\big)$ is a Clifford semigroup.

\smallskip
\item[(d)] $\mathcal B (G \times H)$ is a seminormal monoid.
\end{enumerate}
\end{enumerate}
\end{theorem}

\begin{proof}
1. Suppose that $[(g_1, h_1)\bdot \ldots \bdot (g_{\ell}, h_{\ell})] = [(g'_1, h'_1)\bdot \ldots \bdot (g'_{k}, h'_{k})]$.
Let $T = t_1\bdot \ldots \bdot t_n \in \mathcal F (G) \und T' = t'_1\bdot \ldots \bdot t'_m \in \mathcal F (H)$ be sequences with
\[
  g_1\bdot \ldots \bdot g_{\ell}\bdot T \in \mathcal B (G) \quad \und \quad h_1\bdot \ldots \bdot h_{\ell} \bdot T' \in \mathcal B (H) \,.
\]
We can assume $n \ge m$. Accordingly, we consider the sequence $(t_1, t'_1)\bdot \ldots \bdot (t_m, t'_m)\bdot (t_{m+1}, 1_H)\bdot \ldots \bdot (t_n, 1_H) \in \mathcal F (G\times H)$.
Then we have
\[
  (g_1, h_1)\bdot \ldots \bdot (g_{\ell}, h_{\ell})\bdot (t_1, t'_1)\bdot \ldots \bdot (t_m, t'_m)\bdot (t_{m+1}, 1_H)\bdot \ldots \bdot (t_n, 1_H) \in \mathcal B (G\times H) \,.
\]
It follows that
\[
 (g'_1, h'_1)\bdot \ldots \bdot (g'_{k}, h'_{k}) \bdot (t_1, t'_1)\bdot \ldots \bdot (t_m, t'_m)\bdot (t_{m+1}, 1_H)\bdot \ldots \bdot (t_n, 1_H) \in \mathcal B (G\times H) \,,
\]
and hence $g'_1 \bdot \ldots \bdot g'_k \bdot T \in \mathcal B (G) \und h'_1 \bdot \ldots \bdot h'_k \bdot T' \in \mathcal B (H)$.
By symmetry, the given map is well-defined semigroup homomorphism.

To show surjectivity, let $\big([g_1 \bdot \ldots \bdot g_{\ell}], [h_1 \bdot \ldots \bdot h_k] \big) \in \mathcal C \big(\mathcal B (G), \mathcal F (G)\big) \times \mathcal C \big(\mathcal B (H), \mathcal F (H)\big)$.
If $\ell > k$, then the image of the sequence
\[
  (g_1, h_1)\bdot \ldots \bdot (g_k, h_k)\bdot (g_{k+1},1_H)\bdot \ldots \bdot (g_{\ell}, 1_H) \in \mathcal F (G \times H)
\]
is $\big([g_1 \bdot \ldots \bdot g_{\ell}], [h_1 \bdot \ldots \bdot h_k \bdot 1_H \bdot \ldots \bdot 1_H]\big) = \big([g_1 \bdot \ldots \bdot g_{\ell}], [h_1 \bdot \ldots \bdot h_k] \big)$ by \cite[Lemma 3.6.4]{Oh18a}.
If $\ell < k$, then the image of the sequence
\[
  (g_1, h_1)\bdot \ldots \bdot (g_{\ell}, h_{\ell})\bdot (1_G, h_{\ell +1})\bdot \ldots \bdot (1_G, h_k) \in \mathcal F (G \times H)
\]
is $\big([g_1 \bdot \ldots \bdot g_{\ell}\bdot 1_G \bdot \ldots \bdot 1_G], [h_1 \bdot \ldots \bdot h_k]\big) = \big([g_1 \bdot \ldots \bdot g_{\ell}], [h_1 \bdot \ldots \bdot h_k] \big)$ by \cite[Lemma 3.6.4]{Oh18a}.
It follows that the given map is surjective.

To show injectivity, let
\[
  (g_1, h_1)\bdot \ldots \bdot (g_{\ell}, h_{\ell}) \und (g'_1, h'_1)\bdot \ldots \bdot (g'_{k}, h'_{k}) \in \mathcal F (G\times H)
\]
such that $[g_1 \bdot \ldots \bdot g_{\ell}] = [g'_1 \bdot \ldots \bdot g'_k ] \und [h_1 \bdot \ldots \bdot h_{\ell}] = [h'_1 \bdot \ldots \bdot h'_k ]$.
If $(g_1, h_1)\bdot \ldots \bdot (g_{\ell}, h_{\ell})\bdot (t_1, t'_1)\bdot \ldots \bdot (t_n, t'_n) \in \mathcal B (G\times H)$, then we have that
\[
  g_1\bdot \ldots \bdot g_{\ell}\bdot t_1 \bdot \ldots \bdot t_n \in \mathcal B (G) \und h_1 \bdot \ldots \bdot h_{\ell}\bdot t'_1 \bdot \ldots \bdot t'_n \in \mathcal B (H) \,.
\]
It follows that
\[
  g'_1\bdot \ldots \bdot g'_{k}\bdot t_1 \bdot \ldots \bdot t_n \in \mathcal B (G) \und h'_1 \bdot \ldots \bdot h'_{k}\bdot t'_1 \bdot \ldots \bdot t'_n \in \mathcal B (H) \,,
\]
and thus $(g'_1, h'_1)\bdot \ldots \bdot (g'_k, h'_k)\bdot (t_1, t'_1)\bdot \ldots \bdot (t_n, t'_n) \in \mathcal B (G\times H)$.
By symmetry, we obtain that the given map is injective.

\smallskip
2. By Theorem \ref{3.1}, we have $\mathcal C \big(\mathcal B (H), \mathcal F (H)\big) \cong H$. Thus isomorphism from 1. shows that $\mathcal C \big(\mathcal B (G\times H), \mathcal F (G\times H)\big)$ is a Clifford semigroup if and only if $\mathcal C \big(\mathcal B (G), \mathcal F (G)\big)$ is a Clifford semigroup. The remaining equivalences follow from Theorem \ref{3.6}.
\end{proof}

\smallskip
Now we state the second part of our main result, namely $|G'| \le 2$ if and only if $\mathcal C \big(\mathcal B (G), \mathcal F (G)\big)$ is Clifford. Recall that if $G$ is a finite group with $|G'| \le 2$, then its class semigroup is a Clifford (see \cite[Theorem 3.10]{Oh18a}). To show the converse, we start with the following simple observation.

\smallskip
\begin{lemma} \label{3.10}~
Let $\mathcal C = \mathcal C \big(\mathcal B (G), \mathcal F (G)\big)$. If $\mathcal C$ is a Clifford semigroup, then for every $S \in \mathcal F (G)$, we have $|\pi(S)| \bigm| |G'|$, in particular, $|G'|$ is divided by $2$ if $G$ is non-abelian.
\end{lemma}

\begin{proof}
Let $S \in \mathcal F (G)$ be a sequence. Since $\mathcal C$ is Clifford, there is a $S' \in \mathcal F (G)$ such that $[S'] \in \mathsf E (\mathcal C)$ and $[S] \in \mathcal H \big([S']\big)$. By Proposition \ref{3.3} and \ref{3.5}, we have that $|\pi(S)| = |\pi(S')|$ and it divides the order of $G'$. In particular, if $G$ is non-abelian, then there are $g,h \in G$ such that $gh \neq hg$ and thus $|\pi(g\bdot h)| = 2$.
\end{proof}

\smallskip
\begin{theorem} \label{3.11}~
Let $\mathcal C = \mathcal C \big(\mathcal B (G), \mathcal F (G)\big)$. If $\mathcal C$ is a Clifford semigroup, then $|G'| \le 2$.
\end{theorem}

\begin{proof}
Assume to the contrary that $|G'| \ge 3$.

\smallskip
\begin{itemize}
\item[{\bf A1.}] For every sequence $S \in \mathcal F (G)$, we may suppose that
    \[
      |\pi(S)| \,\, \in \,\, \big\{ 1, \,\, 2, \,\, |G'| \big\} \,.
    \]
\end{itemize}

\begin{proof}[Proof of {\bf A1.}]
Since $G$ is non-abelian, we first observe that the existence of a sequence $S$ satisfying the statement follow by Lemma \ref{3.10} and \cite[Lemma 3.6.5 and 3.7.3]{Oh18a}.
If there is a sequence $S \in \mathcal F (G)$ with $2 < |\pi(S)| < |G'|$, then, since $\mathcal C$ is Clifford, it follows that there is a $S' \in \mathcal F (G)$ such that $[S'] \in \mathsf E (\mathcal C)$ satisfying the condition given in Lemma \ref{3.4}, and hence $|G'_0| = |\pi(S')| = |\pi(S)|$ by Proposition \ref{3.5}, where $G_0 = \la \supp(S') \ra$ is a subgroup of $G$. 
Since $\mathcal B (G)$ is seminormal by Theorem \ref{3.6} and the inclusion $\mathcal B (G_0) \hookrightarrow \mathcal B (G)$ is a divisor homomorphism (see, \cite[Lemma 3.3]{Oh18a}), $\mathcal B (G_0)$ is also seminormal by \cite[Lemma 3.2.4]{Ge-Ka-Re15a}. It follows that $\mathcal C \big(\mathcal B (G_0), \mathcal F (G_0)\big)$ is Clifford by Theorem \ref{3.6} with $|G'_0| = |\pi(S)| \ge 3$, and thus we can replace $G$ with $G_0$. Since $G$ is finite, we obtain the assertion in finite step of this process.
\end{proof}

\smallskip
\begin{itemize}
\item[{\bf A2.}] For every $g \in G$, we have $g^{2} \in \mathsf Z (G)$.
\end{itemize}

\begin{proof}[Proof of {\bf A2.}]
Let $h \in G$. If $gh = hg$, then $g^{2}h = hg^{2}$. Suppose that $gh \neq hg$, We need to show that $g^{2}h = hg^{2}$. Assume to the contrary that $g^{2}h \neq hg^{2}$. Then
\[
  \pi(g \bdot g \bdot h) = \{ ggh, \,\, ghg, \,\, hgg \} \,, \,\, \mbox{ whence } \, |\pi(g \bdot g \bdot h)| = 3 \,.
\]
Then {\bf A1} implies $|G'| = 3$, a contradiction to Lemma \ref{3.10}. Thus $|\pi(g \bdot g \bdot h)| = 2$ and $g^{2}h = hg^{2}$.
\end{proof}

\smallskip
\begin{itemize}
\item[{\bf A3.}] For every $g, h \in G$ with $gh \neq hg$, its commutator $[g,h]$ has order $2$. In particular, we have that
                 \[
                   [g,h] = [g^{-1},h^{-1}] = [h,g] = [h^{-1},g^{-1}] = [h,g^{-1}] = [h^{-1},g] = [g^{-1},h] = [g,h^{-1}] \,.
                 \]
\end{itemize}

\begin{proof}[Proof of {\bf A3.}]
Let $g, h \in G$ with $gh \neq hg$ and let $S = g \bdot h \bdot g^{-1}h^{-1} \in \mathcal F (G)$. Note that, by {\bf A2},
\[
  \begin{aligned}
    hg = hg \quad & \Leftrightarrow \quad hg^{2}g^{-1} = h^{-1}h^{2}g \\
                  & \Leftrightarrow \quad g^{2}hg^{-1} = h^{-1}gh^{2} \\
                  & \Leftrightarrow \quad ghg^{-1}h^{-1} = g^{-1}h^{-1}gh \,,
  \end{aligned}
\]
and it follows that $\pi(S) = \{ ghg^{-1}h^{1}, \,\, hg^{-1}h^{-1}g, \,\, 1_G \}$.
Since $h^{-1}g^{2} = g^{2}h^{-1}$ by {\bf A2}, it follows that $hg^{-1}h^{-1}g = hgh^{-1}g^{-1}$. Thus, by {\bf A1},
\[
  \pi(S) = \{ ghg^{-1}h^{-1}, \,\, hgh^{-1}g^{-1}, \,\, 1_G \} \,, \mbox{ whence } [g,h] \, = \, [g,h]^{-1} \, = \, [h,g] \,.
\]
Moreover, we obtain that $[g^{-1},h^{-1}] = [g,h] = [h,g] = [h,g^{-1}]$. By swiping the role between $g$ and $h$, we have that
\[
  [h^{-1},g^{-1}] \, = \, [h,g] \, = \, [g,h] \, = \, [g,h^{-1}] \,.
\]
Since $g^{2}h = hg^{2}$ and $h^{2}g = gh^{2}$ by {\bf A2}, we also obtain that
\[
  [h^{-1},g] \, = \, [h^{-1}, g^{-1}] \und [g^{-1},h] \, = \, [g^{-1},h^{-1}] \,.
\]
Therefore, $[g,h] = [g^{-1},h^{-1}] = [h,g] = [h^{-1},g^{-1}] = [h,g^{-1}] = [h^{-1},g] = [g^{-1},h] = [g,h^{-1}]$.
\end{proof}

\smallskip
\begin{itemize}
\item[{\bf A4.}] $G$ is finite group of nilpotency class $2$.
\end{itemize}

\begin{proof}[Proof of {\bf A4.}]
It is well known that a group $G$ has nilpotency class $2$ if and only if $G' \subset \mathsf Z (G)$ (see \cite[Theorem 6.1.8]{Du-Fo04}). Hence it is sufficient to show $G' \subset \mathsf Z (G)$. 
Let $H = \la g^{2} \t g \in G \ra$ be a subgroup of $G$. Then $H \subset \mathsf Z (G)$ by {\bf A2}, and hence $H$ is normal subgroup of $G$. Since every elements in $G/H$ has order at most $2$, $G/H$ is an abelian group. It follows that $G' \subset H \subset \mathsf Z (G)$.
\end{proof}

\smallskip
\begin{itemize}
\item[{\bf A5.}] There is a $2$-subgroup $\overline{G} \subset G$ such that $\mathcal C \big(\mathcal B ( \overline{G} ), \mathcal F ( \overline{G} )\big)$ is Clifford, $\overline{G} \,'$ is not cyclic with $|\overline{G} \,'| \ge 3$, and each proper subgroup $H$ of $\overline{G}$ has a cyclic commutator subgroup $H'$.
\end{itemize}

\begin{proof}[Proof of {\bf A5.}]
Since $G$ is finite nilpotent group by {\bf A4}, $G$ is a direct product of its Sylow subgroups (see \cite[Theorem 6.1.3]{Du-Fo04}). For a prime number $p$, we denote by $Syl_p (G)$ a Sylow $p$-subgroup of $G$. Then Lemma \ref{3.10} implies that $Syl_p (G)$ is abelian for every odd prime $p$, and it follows that $|Syl_2 (G)'| = |G'| \ge 3$ because the commutator subgroup of direct product is the direct product of each commutator subgroups. Then Theorem \ref{3.9}.2 implies that we can replace $G$ with $G_0 = Syl_2 (G)$ so that $\mathcal C \big(\mathcal B (G_0), \mathcal F (G_0)\big)$ is Clifford and $G_0$ satisfies {\bf A1} - {\bf A4}.
Let $\Omega = \{ \overline{G} \, \t \, \overline{G} \subset G_0 \, \mbox{ is a subgroup with } \, \overline{G} \,' = G'_0 \}$. Since $G_0$ is finite, there is a minimal element $\overline{G}$ in $\Omega$. By replacing $G_0$ with $\overline{G}$, we can assume that $\overline{G}$ is a $2$-group with $|\overline{G} \,'| \ge 3$ such that $\mathcal C \big(\mathcal B ( \overline{G} ), \mathcal F ( \overline{G} )\big)$ is Clifford and $\overline{G}$ satisfies {\bf A1} - {\bf A4}. Since $\overline{G} \,'$ is abelian group generated by order $2$ elements by {\bf A3} and $|\overline{G} \,'| \ge 3$, we have that $\overline{G} \,'$ is not cyclic. By the minimality of $\overline{G}$, we obtain that, for each proper subgroup $H$, $H'$ is cyclic (indeed, {\bf A1} implies $|H'| \le 2$).
\end{proof}

\smallskip
For simplicity of notation, we set $G = \overline{G}$ and suppose that $G$ has all properties listed in {\bf A5}. Then \cite[Theorem 139.A]{Be-Ja11a} implies that $G = \la g_1, g_2, g_3 \ra$, and $G' \cong C_2 \oplus C_2 \oplus C_2$ or $C_2 \oplus C_2$.
Note that $G' = \big\la [g_1,g_2], [g_1,g_3], [g_2,g_3] \big\ra$, and if $x,y,z \in G$ are any elements, then since $G' \subset \mathsf Z (G)$, we obtain that
\begin{equation} \label{comm}
[x,y][x,z] \, = \, (xyx^{-1}y^{-1})[x,z] \, = \, xyx^{-1}[x,z]y^{-1} \, = \, x(yz)x^{-1}(yz)^{-1} \, = \, [x,yz] \,.
\end{equation}

\bigskip
\noindent CASE 1{\rm \,:} $G' \cong C_2 \oplus C_2 \oplus C_2$.

Then $[g_i, g_j]$ are non-trivial distinct elements of order $2$ for all $i,j \in [1,3]$ with $i \neq j$. We assert that the elements in the set
\[
  \{ g_2g_1g_3, \,\, g_2g_3g_1, \,\, g_3g_2g_1 \}
\]
are all distinct. Since $[g_1,g_3]$ and $[g_2,g_3]$ are non-trivial, it follows that
\[
  g_2g_1g_3 \neq g_2g_3g_1, \und g_2g_3g_1 \neq g_3g_2g_1 \,.
\]
If $g_2g_1g_3 = g_3g_2g_1$, then it follows by (\ref{comm}) that
\[
  1_G \, = \, (g_2g_1)g_3(g^{-1}_1g^{-1}_2)g^{-1}_3 \, = \, [g_2g_1, g_3] \, = \, [g_2,g_3][g_1,g_3] \,,
\]
and thus $[g_2,g_3] = [g_1,g_3]$ by {\bf A3}, a contradiction. Hence we have $g_2g_1g_3 \neq g_3g_2g_1$, and it follows that $3 \le |\pi(g_1 \bdot g_2 \bdot g_3)| \le 6$, a contradiction to {\bf A1}.

\bigskip
\noindent
CASE 2{\rm \,:} $G' \cong C_2 \oplus C_2$.

Without loss of generality, we can assume that $[g_1,g_2]$, $[g_1,g_3]$ are non-trivial distinct elements of order $2$. Then we assert that there is a sequence $S \in \mathcal F (G)$ with $|\pi(S)| = 3$, a contradiction to {\bf A1}.

\medskip
\noindent
CASE 2.1{\rm \,:} $[g_2,g_3] = 1_G$.

Let $S = g_1 \bdot g_2g_1 \bdot g_3g_1 \in \mathcal F (G)$. Then $g_2g_1 \neq 1_G$ and $g_3g_1 \neq 1_G$, otherwise $[g_1,g_2] = [g_1,g_3] = 1_G$.
Since $[g_1,g_2]$ is non-trivial, it follows that $g_1(g_2g_1)(g_3g_1) \neq (g_2g_1)g_1(g_3g_1)$.

If $g_1(g_2g_1)(g_3g_1) = g_1(g_3g_1)(g_2g_1)$, then we have $g_2g_1g_3 = g_3g_1g_2$, and since $[g_2,g_3] = 1_G$, it follows that
\[
  1_G \, = \, (g_2g_1g_3)(g^{-1}_2g^{-1}_1g^{-1}_3) \, = \, g_2g_1(g^{-1}_2g_3)g^{-1}_1g^{-1}_3 \, =  \,[g_2,g_1][g_1,g_3] \,,
\]
whence $[g_1,g_3] = [g_2,g_1] = [g_1,g_2]$ by {\bf A3}, a contradiction. Thus $g_1(g_2g_1)(g_3g_1) \neq g_1(g_3g_1)(g_2g_1)$.

If $g_1(g_3g_1)(g_2g_1) = (g_2g_1)g_1(g_3g_1)$, then we have $g_3g_1g_2 = g_1g_2g_3$ because $g^{2}_1 \in \mathsf Z (G)$ by {\bf A2}, and since $[g_2,g_3] = 1_G$, it follows that $g_3g_1 = g_1g_3$, a contradiction. Thus $g_1(g_3g_1)(g_2g_1) \neq (g_2g_1)g_1(g_3g_1)$.

Since $[g_2,g_3] = 1_G$, by (\ref{comm}),
\[
  \begin{aligned}
  \, [g_1,g_3] \, = \, [g_3,g_1] \quad & \Leftrightarrow \quad [g_1,g_3] \, = \, [g_3,g_1][g_3,g_2] \, = \, [g_3, g_1g_2] \\
                                       & \Leftrightarrow \quad g_1g_2[g_1,g_3] \, = \, [g_1,g_3]g_1g_2 \, = \, g_3(g_1g_2)g^{-1}_3 \\
                                       & \Leftrightarrow \quad g_1g_2(g_1g_3g^{-1}_1g^{-1}_3) \, = \, g_3g_1g_2g^{-1}_3 \\
                                       & \Leftrightarrow \quad g_1g_2g_1g_3 \, = \, g_3g_1g_2g_1 \\
                                       & \Leftrightarrow \quad g_1(g_2g_1)(g_3g_1) \, = \, (g_3g_1)(g_2g_1)g_1 \,,
  \end{aligned}
\]
and
\[
  \begin{aligned}
  \, [g_1,g_2] \, = \, [g_2,g_1] \quad & \Leftrightarrow \quad [g_3g_1,g_2] \, = \, [g_3,g_2][g_1,g_2] \, = \, [g_2,g_1] \quad  \quad \,\\
                                       & \Leftrightarrow \quad g_1[g_3g_1,g_2] \, = \, [g_3g_1,g_2]g_1 = g_2g_1g^{-1}_2 \\
                                       & \Leftrightarrow \quad g_1(g_3g_1)g_2(g^{-1}_1g^{-1}_3)g^{-1}_2 \, = \, g_2g_1g^{-1}_2 \\
                                       & \Leftrightarrow \quad g_1(g_3g_1)(g_2g_1) \, = \, (g_2g_1)(g_3g_1)g_1 \,.
  \end{aligned}
\]
Since $g_2g_3 = g_3g_2$ and $g^{2}_1 \in \mathsf Z (G)$ by {\bf A2}, it follows that $(g_2g_1)g_1(g_3g_1) = (g_3g_1)g_1(g_2g_1)$.
Thus we obtain
\[
  \pi(S) = \{ g_1(g_2g_1)(g_3g_1), \,\, g_1(g_3g_1)(g_2g_1), \,\, (g_2g_1)g_1(g_3g_1) \} \,.
\]

\medskip
\noindent
CASE 2.2{\rm \,:} $[g_2,g_3] = [g_1,g_2]$.

By changing a generator of $G$, we can assume $G = \la g_1, g_2, x \ra$, where $x = g_3g^{-1}_1$. Then, by (\ref{comm}),
\[
  [g_1,x] \, = \, [g_1,g_3g^{-1}_1] \, = \, [g_1, g_3][g_1,g ^{-1}_1] \, = \, [g_1,g_3] \,,
\]
and by {\bf A3}, we also have that
\[
  [g_2,x] \, = \, [g_2,g_3g^{-1}_1] \, = \, [g_2, g_3][g_2,g^{-1}_1] \, = \, [g_1,g_2][g_1,g_2] \, = \, 1_G \,.
\]
Thus CASE 2.1 implies that $S = g_1 \bdot g_2g_1 \bdot xg_1 \in \mathcal F (G)$ is the desired sequence.

\medskip
\noindent
CASE 2.3{\rm \,:} $[g_2,g_3] = [g_1,g_3]$.

Similarly, we can assume $G = \la g_1, x, g_3 \ra$ ,where $x = g_2g^{-1}_1$. Then, by (\ref{comm}),
\[
  [g_1,x] \, = \, [g_1,g_2g^{-1}_1] \, = \, [g_1,g_2][g_1,g^{-1}_1] \, = \, [g_1,g_2] \,,
\]
and by {\bf A3}, we also have that
\[
  [x,g_3] \, = \, [g_2g^{-1}_1,g_3] \, = \, [g_2,g_3][g^{-1}_1,g_3] \, = \, [g_1,g_3][g_1,g_3] \, = \, 1_G \,.
\]
Thus CASE 2.1 implies that $S = g_1 \bdot xg_1 \bdot g_3g_1 \in \mathcal F (G)$ is the desired sequence.

\medskip
\noindent
CASE 2.4{\rm \,:} $[g_2,g_3] = [g_1,g_2g_3]$.

Similarly, we can assume $G = \la g_1, x, y \ra$, where $x = g_2g^{-1}_1$, $y = g_3g^{-1}_1$. Then, by (\ref{comm}),
\[
  \begin{aligned}
  \, [g_1,x]  \, & = \, [g_1,g_2g^{-1}_1] \, = \, [g_1,g_2][g_1,g^{-1}_1] \, = \, [g_1,g_2] \und \\
  \, [g_1,y]  \, & = \, [g_1,g_3g^{-1}_1] \, = \, [g_1,g_3][g_1,g^{-1}_1] \, = \, [g_1,g_3] \,,
  \end{aligned}
\]
and by {\bf A3}, we also have that
\[
  [x,y] \, = \, [x,g_3g^{-1}_1] \, = \, [x,g_3][x,g^{-1}_1] \, = \, [g_2g^{-1}_1,g_3][g_2g^{-1}_1,g_1] \, = \, [g_2,g_3][g^{-1}_1,g_3][g_2,g_1] \, = \, 1_G \,.
\]
Thus CASE 2.1 implies that $S = g_1 \bdot xg_1 \bdot yg_1 \in \mathcal F (G)$ is the desired sequence.
\end{proof}

\smallskip
Now we state our main result, which is the generalization of Theorem \ref{3.1}, to sum up Theorem \ref{3.6} and \ref{3.11}.

\smallskip
\begin{corollary} \label{3.12}~
The following statements are equivalent{\rm \,:}
\begin{enumerate}
\item[(a)] $|G'| \le 2$.

\smallskip
\item[(b)] $\mathcal B (G)$ is a seminormal monoid.

\smallskip
\item[(c)] $\mathcal C = \mathcal C \big(\mathcal B (G), \mathcal F (G)\big)$ is a Clifford semigroup.
\end{enumerate}
If this is the case, then $\mathcal B (G) = \big\{ S \in \mathcal F (G) \t [S] \in \mathsf E (\mathcal C) \big\}$.
\end{corollary}

\begin{proof}
(b) $\Leftrightarrow$ (c) follows by Theorem \ref{3.6}, and (a) $\Leftrightarrow$ (c) follows by \cite[Theorem 3.10.2]{Oh18a} and Theorem \ref{3.11}. If this is the case, then \cite[Theorem 3.10.2]{Oh18a} and item 4 in Theorem \ref{3.6} imply the assertion.
\end{proof}

\smallskip
\begin{remark} \label{3.13}~
\begin{enumerate}
\item Suppose that $G_1$, $G_2$ are finite groups with $|G'_1| = |G'_2| = 2$. Then, for each $i \in [1,2]$, we have that $\mathcal C \big(\mathcal B (G_i), \mathcal F (G_i)\big)$ is Clifford by Corollary \ref{3.12}. However, since $(G_1 \times G_2)' = G'_1 \times G'_2$, Corollary \ref{3.12} implies that $\mathcal C \big(\mathcal B (G_1 \times G_2), \mathcal F (G_1 \times G_2)\big)$ is not Clifford (cf. Theorem \ref{3.9}).

\smallskip
\item In general, the converse of Corollary \ref{3.8}.2 does not hold, for example let $G = G_1 \times G_2$ and $N = G'$ where each $G_i$ is a finite group with $|G'_i| = 2$.
\end{enumerate}
\end{remark}

\smallskip
\begin{lemma} \label{3.14}~
Suppose that $|G'| = 2$. Let $g, h \in G$ be the elements with $n = \ord (g)$ and $m = \ord (h)$.
\begin{enumerate}
\item If $n$ is odd, then $g \in \mathsf Z (G)$.

\smallskip
\item We have $g \bdot g^{-1} \sim g^{[n]}$.

\smallskip
\item If $gh = hg$, then $\big[ g^{[n]} \big] \und \big[ h^{[m]} \big]$ are comparable with respect to the Rees order $\le$, and the following conditions are equivalent{\rm \,:}
\begin{enumerate}
\item[(a)] $g^{[n]} \sim h^{[m]}$.

\smallskip
\item[(b)] For any $x \in G$, we have that $xg \neq gx$ if and only if $xh \neq hx$.

\smallskip
\item[(c)] $gh \in \mathsf Z (G)$.
\end{enumerate}
\end{enumerate}
\end{lemma}

\begin{proof}
1. Since $G'$ is a normal subgroup, it follows that, for every $x \in G$, we have $xG' = G'x$. Then, since $|G'| = 2$, we obtain that $G' \subset \mathsf Z (G)$, and it follows that any even power of the elements in $G$ is central (see (\ref{comm})). Since $n$ is odd, it follows that $g^{n-1} \in \mathsf Z (G)$, whence $g^{n-1} \bdot g \, \sim \, g^{n}$ by \cite[Lemma 3.6.4]{Oh18a}. Since $g^{n} = 1_G$ and $1_G \sim 1_{\mathcal F (G)}$, \cite[Lemma 3.7.2]{Oh18a} implies that $g^{n-1}\bdot g \in \mathcal B \big(\mathsf Z (G)\big)$ whence $g \in \mathsf Z (G)$.

\smallskip
2. Observe that if $g \in \mathsf Z (G)$, then the assertion follows by \cite[Lemma 3.7.2]{Oh18a}. Suppose that $g \notin \mathsf Z (G)$. Then item 1 implies that $n$ is even, and hence it is sufficient to show that $g^{n-1} \notin \mathsf Z (G)$. Indeed, if this holds true, then the assertion follows by \cite[Lemma 3.9.2]{Oh18a}. Assume to the contrary that $g^{n-1} \in \mathsf Z (G)$. As at the proof of item 1, we conclude that $g \in \mathsf Z (G)$, a contradiction.

\smallskip
3. By Corollary \ref{3.12}, $\big[ g^{[n]} \big]$, $\big[ h^{[m]} \big] \in \mathsf E (\mathcal C)$. We need to show either $g^{[n]} \bdot h^{[m]} \sim g^{[n]}$ or $g^{[n]} \bdot h^{[m]} \sim h^{[m]}$. 
Clearly if $T \in \mathcal F (G)$ is such that either $T \bdot g^{[n]} \in \mathcal B (G)$ or $T \bdot h^{[m]} \in \mathcal B (G)$, then we have $T \bdot g^{[n]} \bdot h^{[m]} \in \mathcal B (G)$. 
For the converse, let $T \in \mathcal F (G)$ be a sequence with $T \bdot g^{[n]} \bdot h^{[m]} \in \mathcal B (G)$. If $\pi(T \bdot g^{[n]} \bdot h^{[m]}) = \{ 1_G \}$, then $T \in \mathcal B (G)$ and hence we have that $T \bdot g^{[n]} \und T \bdot h^{[m]} \in \mathcal B (G)$. If $\pi \big(T \bdot g^{[n]} \bdot h^{[m]} \big) = G'$, then there are following cases.

\smallskip
\noindent
CASE 1{\rm \,:} \, $|\pi(T)| = 2$.

Since $\pi(T) \subset \pi \big(T \bdot g^{[n]} \bdot h^{[m]} \big) = G'$, it follows that $1_G \in G' = \pi(T)$ whence $T \bdot g^{[n]} \und T \bdot h^{[m]} \in \mathcal B (G)$.

\smallskip
\noindent
CASE 2{\rm \,:} \, $|\pi(T)| = 1$ and there is an $t \in \supp(T)$ such that either $tg \neq gt$ or $th \neq ht$.

If $tg \neq gt$, then $|\pi \big(T \bdot g^{[n]} \big)| = 2$, and since $\pi \big(T \bdot g^{[n]} \big) \subset \pi \big(T \bdot g^{[n]} \bdot h^{[m]} \big) = G'$, it follows that $1_G \in G' = \pi \big(T \bdot g^{[n]} \big)$ whence $T \bdot g^{[n]} \in \mathcal B (G)$. By symmetry, if $th \neq ht$, then $T \bdot h^{[m]} \in \mathcal B (G)$.

\noindent
Thus, in any cases, we obtain either $T \bdot g^{[n]} \in \mathcal B (G)$ or $T \bdot h^{[m]} \in \mathcal B (G)$.

\smallskip
Now we show the equivalent conditions. 

\smallskip
(a) $\Rightarrow$ (b) For any $x \in G$, (a) implies that $x \bdot g^{[n]} \sim x \bdot h^{[m]}$ and hence the assertion follows.

\smallskip
(b) $\Rightarrow$ (a) According to the proof of the main part, we obtain that $g^{[n]} \, \sim \, g^{[n]} \bdot h^{[m]} \, \sim \, h^{[m]}$.

\smallskip
(b) $\Leftrightarrow$ (c) It follows by the equation $[x,gh] = [x,g][x,h]$ because $G' \subset \mathsf Z (G)$ (see (\ref{comm})).
\end{proof}

\smallskip
When $G$ satisfies the equivalent condition given in Corollary \ref{3.12}, we can obtain more precise structural information of the maximal subgroups of the class semigroup as following (cf. Proposition \ref{3.5}).

\smallskip
\begin{proposition} \label{3.15}~
Suppose that $|G'| = 2$, and we set $\mathcal C = \mathcal C \big(\mathcal B (G), \mathcal F (G)\big)$.
If $S \in \mathcal F (G)$ is a sequence with $[S] \in \mathsf E (\mathcal C)$ satisfying the condition given in \textnormal{Lemma \ref{3.4}}, then there is an isomorphism from $\mathcal H \big([S]\big)$ to $G_0/G'_0$, where $G_0 = \la \supp(S) \ra \subset G$ is a subgroup. In particular, we have the following cases{\rm \,:}
\begin{enumerate}
\item If $\pi(S) = G'$, then $\mathcal H \big([S]\big)$ is isomorphic to $G/G'$.

\smallskip
\item If $[S] = [1_{\mathcal F (G)}]$, then $\mathcal H \big([S]\big)$ is isomorphic to $\mathsf Z (G)$.

\smallskip
\item If $\pi(S) = \{ 1_G \}$ with $[S] \neq [1_{\mathcal F (G)}]$, then $\mathcal H \big([S]\big)$ is isomorphic to an abelian subgroup $G_0$ of $G$. Moreover, if there is no $h \in G$ such that $h$ commute with all elements in $\supp(S)$ and $hg \notin \mathsf Z (G)$ for every $g \in \supp(S)$, then $G_0$ is a maximal among the abelian subgroups of $G$.
\end{enumerate}
\end{proposition}

\begin{proof}
1. and 2. This follows from Proposition \ref{3.5}.

\smallskip
3. Let $S = g_1 \bdot \ldots \bdot g_{\ell} \in \mathcal F (G)$ be such that $\pi(S) = \{1_G\}$ and $[S] \neq [1_{\mathcal F (G)}]$. Then, by Proposition \ref{3.3}, $G_0 = \la \supp(S) \ra = \la g_1, \ldots, g_{\ell} \ra \subset G$ is a subgroup with $G'_0 = \pi(S) = \{ 1_G \}$, whence it is abelian. Let $n_i = \ord (g_i)$ for $ i \in [1,\ell]$ and let $m = \lcm \{ n_i \t i \in [1,\ell] \}$. Then, by Proposition \ref{3.5}, the map
\[
  \begin{aligned}
    \varphi_{[S]} : \mathcal H \big([S]\big) \quad & \to \quad G_0 \\
                                         [T] \quad & \mapsto \quad \pi(T)
  \end{aligned}
\]
is a group homomorphism. We assert that this map is bijective.
To show injectivity, let $T \in \mathcal F (G)$ be a sequence with $[T] \in \ker (\varphi_{[S]})$. Then $\pi(T) = \{ 1_G \}$ whence $T \in \mathcal B (G)$. By Corollary \ref{3.12}, we have $[T] \in \mathsf E (\mathcal C)$, and it follows that $[T] = [S]$ by Lemma \ref{3.2}.1. Since $[S] \in \mathcal H \big([S]\big)$ is identity element by Lemma \ref{3.2}.1, we obtain that $\varphi_{[S]}$ has the trivial kernel whence it is injective.
To show surjectivity, it is sufficient to show that for every $g \in G$ there is a $T \in \mathcal F (G)$ such that $[T] \in \mathcal H \big([S]\big)$ with $\pi(T) = \{ g \}$.
Let $g \in G_0$ so that $g = g^{r_1}_1\ldots g^{r_{\ell}}_{\ell}$, where $r_1, \ldots, r_{\ell} \in \N_0$ with $r_i \in [1,n_i]$ for $i \in [1,\ell]$.
Then we consider a sequence 
\[
  T = S \bdot g^{[r_1]}_1 \bdot g^{[r_2]}_2 \bdot \ldots \bdot g^{[r_{\ell}]}_{\ell} \in \mathcal F (G) \,.
\]
Since $[S] \in \mathsf E (\mathcal C)$,
\[
  [T] + [S] \, = \, [T] \quad \und \quad [T] + \big[ g^{[m - r_1]}_1 \bdot \ldots \bdot g^{[m - r_{\ell}]}_{\ell} \big] \, = \, \big[S \bdot g^{[m]}_1 \bdot \ldots \bdot g^{[m]}_{\ell}\big] \, = \, \big[S^{[m+1]}\big] \, = \, [S] \,,
\]
whence $[T] \in \mathcal H \big([S]\big)$ and $\pi(T) = \{ g \}$ by the construction. Thus $T \in \mathcal F (G)$ is the desired sequence and hence $\varphi_{[S]}$ is surjective.

Suppose in addition that there are no such elements stated in $3$. Since $G_0 = \la \supp(S) \ra$ is abelian, Lemma \ref{3.14}.3 implies that the subset $\big\{ \big[ g^{[n_i]}_i \big] \t i \in [1,\ell] \big\} \subset \mathsf E (\mathcal C)$ is totally ordered, and hence after renumbering if necessary, we can assume that 
\[
  \big[ g^{[n_1]}_1 \big] \,\, \le \,\, \ldots \,\, \le \,\, \big[ g^{[n_{\ell}]}_{\ell} \big] \,.
\]
Since $S = g_1 \bdot \ldots \bdot g_{\ell}$ with $[S] \in \mathsf E (\mathcal C)$, it follows that
\[
  S \, \sim \, S^{[m]} \, \sim \, g^{[n_1]}_1 \bdot \ldots \bdot g^{[n_{\ell}]}_{\ell} \, \sim \, g^{[n_1]}_1 \,.
\]
Now let $H \subset G$ be an abelian subgroup containing $G_0$, and let $h \in H$ with $k = \ord (h)$. Then, by assumption, there is an $i \in [1,\ell]$ such that $hg_i \in \mathsf Z (G)$. Hence Lemma \ref{3.14}.3 implies that $h^{[k]} \sim g^{[n_i]}_i$. Since $S \sim g^{[n_1]}_1$ and $\big[g^{[n_1]}_1\big] \le \big[g^{[n_j]}_j\big]$ for all $j \in [1,\ell]$, it follows that $h^{[k]}\bdot S \sim S$, and hence $h \in \la \supp(S) \ra = G_0$ by Lemma \ref{3.4}. Thus we obtain that $G_0 = H$ whence $G_0$ is a maximal abelian subgroup of $G$.
\end{proof}

\bigskip
\section{Arithmetic of the monoid of product-one sequences} \label{4}
\bigskip

Transfer Krull monoids over finite abelian groups include commutative Krull monoids with finite class group having prime divisors in all classes (including, in particular, commutative Krull and Dedekind domains) but they also include wide classes of non-commutative Dedekind domains (see \cite{Sm13a, Ba-Sm15, Sm18a} for original work and \cite{Ge16c} for a survey and extended list of examples). Let $H$ be a transfer Krull monoid and $\theta \colon H \to \mathcal B (G)$ a (weak) transfer homomorphism to the monoid of product-one sequences over a finite abelian group $G$. Then sets of lengths of $H$ and of $\mathcal B (G)$ coincide and because of this connection the study of sets of lengths of $\mathcal B (G)$ is a central topic in factorization theory.
The  Characterization Problem is one of the most important problems in this area. It asks whether for each two transfer Krull monoids $H$ and $H'$ over groups $G$ and $G'$  their system of sets of lengths coincide if and only if $G$ and $G'$ are isomorphic. The standing conjecture is that this holds true for all finite abelian groups $G$ (apart from two trivial exceptional pairings) and  we refer to \cite{Ge-Sc16a, Ge-Zh17b, Zh17a, Zh19a} for recent progress in this direction.

In this section we study sets of lengths of $\mathcal B (G)$ for non-abelian groups. Our goal is to understand if and to what extent their sets of lengths differ from sets of lengths of $\mathcal B (G)$ over finite abelian groups. We briefly gather terminology and notation.

Let $H$ be an atomic monoid and $a, b \in H$. If $a = u_1 \cdot \ldots \cdot u_k$, where $k \in \N$ and $u_1, \ldots, u_k \in \mathcal A (H)$, then $k$ is called the length of the factorization and $\mathsf L (a) = \{ k \in \N \mid a \ \text{has a factorization of length} \ k \} \subset \N$ is the {\it set of lengths} of $a$. As usual we set $\mathsf L (a) = \{0\}$ if $a \in H^{\times}$, and then
\[
  \mathcal L (H) = \{ \mathsf L (a) \mid a \in H \}
\]
denotes the {\it system of sets of lengths} of $H$. If $k \in \N$ and  $H \ne H^{\times}$, then
\[
  \mathcal U_k (H) = \bigcup_{k \in L, L \in \mathcal L (H)} L \ \subset \ \N
\]
denotes the union of sets of lengths containing $k$, and we set $\rho_k (H) = \sup \mathcal U_k (H)$.
If $L = \{m_1, \ldots, m_k \} \subset \Z$ is a finite subset of the integers, where $k \in \N$ and $m_1 < \ldots < m_k$, then $\Delta (L) = \{m_i - m_{i-1} \mid i \in [2,k]\} \subset \N$ is the set of distances of $L$. If all sets of lengths are finite, then
\[
  \Delta (H) = \bigcup_{L \in \mathcal L (H)} \Delta (L)
\]
is the {\it  set of distances} of $H$.  Let $G$ be a finite group. As usual, we set
\[
\Delta (G) = \Delta \big(\mathcal B (G) \big), \mathcal L (G) = \mathcal L \big( \mathcal B (G) \big), \quad \text{and} \quad \rho_k (G) = \rho_k \big( \mathcal B (G) \big) \quad \text{for every $k \in \N$} \,.
\]
Since $\mathcal B (G)$ is finitely generated, it is easy to see that $\Delta (G)$ and all $\mathcal U_k (G)$ are finite (see \cite[Theorem 3.1.4]{Ge-HK06a}). Furthermore, we have $\mathcal U_k (G)=\{k\}$ for all $k \in \N$ (equivalently, $\Delta (G)=\emptyset$) if and only if $|G|\le 2$ (see \cite[Theorem 3.2.4]{Cz-Do-Ge16a}). Thus, whenever convenient, we will assume that $|G|\ge 3$.

Along the lines of the proofs in the abelian setting we showed in \cite[Theorem 5.5]{Oh18a} that $\mathcal U_k (G)$ is a finite interval for all $k \in \N$ and, under a mild additional hypothesis, that $\Delta (G)$ is a finite interval.
However, there are also striking differences between the abelian and the non-abelian setting and they all have their origin in the fact that, in non-abelian case, the embedding $\mathcal B (G) \hookrightarrow \mathcal F (G)$ is not a divisor homomorphism (see Theorem \ref{3.1}). Thus there exist $U, V \in \mathcal B (G)$ such that $U$ divides $V$ in $\mathcal F (G)$, but not in $\mathcal B (G)$. Moreover, $U$ and $V$ can be atoms (e.g., if $G$ is the quaternion group, then $U=I^{[4]} \in \mathcal A (G)$ and $V=I^{[4]} \bdot J^{[2]} \in \mathcal A (G)$ have this property \cite[Example 4.1]{Oh18a}).

Suppose that $G$ is abelian. The Characterization Problem asks whether for any finite abelian group $G^*$, $\mathcal L (G)= \mathcal L (G^*)$ implies that $G$ and $G^*$ are isomorphic. The only exceptional cases known so far are groups with small Davenport constant. Indeed, we have $\mathcal L (C_1) = \mathcal L (C_2)$ and $\mathcal L (C_3) = \mathcal L (C_2 \oplus C_2)$ (note that these four groups are precisely the groups having Davenport constant at most three). All groups having Davenport constant at most five are abelian. In Theorem \ref{4.7}, we show that if $G$ is a finite group with Davenport constant six and $G^*$ is any finite group with $\mathcal L (G) = \mathcal L (G^*)$, then $G$ and $G^*$ are isomorphic. Results of the same flavor are established in Theorem \ref{4.4}.

Before finding sets of lengths which are characteristic for a given group, we determine those sets of non-negative integers which are sets of lengths over all finite groups. It turns out that this is a simple consequence of the associated result in the abelian setting (sets which are sets of lengths in all numerical monoids are determined in \cite{Ge-Sc19a}).

\smallskip
\begin{lemma} \label{4.1}~
We have
\[
   \bigcap \  \mathcal L (G) = \{ y + 2k + [0,k] \t y, k \in \N_0 \} \,,
\]
where the intersection is taken over all finite groups $G$ with $|G| \ge 3$.
\end{lemma}

\begin{proof}
From \cite[Theorem 3.6]{Ge-Sc-Zh17b}, we obtain that
\[
  \bigcap_{(1)} \mathcal L (G) \,\, \subset \,\, \bigcap_{(2)} \mathcal L (G) \overset{(a)}{=} \{ y + 2k + [0,k] \t y, k \in \N_0 \} \,,
\]
where the  intersection (1) is taken over all finite groups with $|G|\ge 3$ and the  intersection (2) is taken over all finite abelian groups $G$ with $|G|\ge 3$.
Hence it suffices to show that for every finite  group $G$ with $|G|\ge 3$ we have
\[
  \{ y + 2k + [0,k] \t y, k \in \N_0 \} \subset \mathcal L (G) \,.
\]
Let $G$ be a finite  group with $|G|\ge 3$. Then $G$ contains an element $g \in G$ with $\ord(g) = n \ge 3$ or $G$ contains two distinct elements of order two that commute with each other. Then Equation (a) implies that
\[
  \{ y + 2k + [0,k] \t y, k \in \N_0 \} \subset \mathcal L (C_n) \cap \mathcal L (C_2 \oplus C_2) \subset \mathcal L (G) \,. \qedhere
\]
\end{proof}

\smallskip
For a sequence $S = g_1 \bdot \ldots \bdot g_{\ell} \in \mathcal F (G)$, we use the following notation{\rm \,:}
\[
  S^{-1} = g^{-1}_1 \bdot \ldots \bdot g^{-1}_{\ell} \in \mathcal F (G) \,.
\]

\smallskip
\begin{lemma} \label{4.2}~
Let $G$ be a finite group with $|G| \ge 3$.
\begin{enumerate}
\item The following statements are equivalent{\rm \,:}
       \begin{enumerate}
        \item[(a)] $G$ is non-abelian.

        \smallskip
        \item[(b)] Every sequence $S \in \mathcal F (G)$ of length $\mathsf D (G)$ has a non-trivial proper product-one subsequence $($i.e., there is a sequence $T \in \mathcal B (G)$ with $1_{\mathcal F (G)} \ne T \neq S$ and $T \t S )$.
       \end{enumerate}

\smallskip
\item We have $\mathsf D (G/G') \le \mathsf D (G)$.

\smallskip
\item For every $S \in \mathcal B (G)$ the following statements are equivalent{\rm \,:}
      \begin{enumerate}
      \item[(a)] $\{2, \mathsf D (G)\} \subset \mathsf L (S)$.

      \smallskip
      \item[(b)] $S = U\bdot U^{-1}$ for some $U \in \mathcal A (G)$ with $| U | = \mathsf D (G)$.
\end{enumerate}
\end{enumerate}
\end{lemma}

\begin{proof}
1. (a) $\Rightarrow$ (b) Assume to the contrary that there exists a $S \in \mathcal F (G)$ of length $\mathsf D (G)$ such that $S$ has no non-trivial proper product-one subsequence.
If $S$ is not a product-one sequence, then $S$ is  product-one free whence $\mathsf D (G) = |S| \le \mathsf d (G) < \mathsf D (G)$, a contradiction. Thus $S$ is a product-one sequence and hence it is an atom. If we set $S = g_1 \bdot \ldots \bdot g_{\mathsf D (G)}$, then $T = g_1 \bdot \ldots \bdot g_{\mathsf D (G) - 1}$ is  product-one free of length $|T| = \mathsf D (G) - 1$. Since $\mathsf d (G) + 1 \le \mathsf D (G)$, we have that $\mathsf D (G) = \mathsf d (G) + 1$ and $|T| = \mathsf d (G)$, and hence $\la \supp(S) \ra = G$ by Lemma \ref{2.2}.2.

We now assert that $\pi(S) = \{ 1_G \}$. If there exists an $h \in \pi(S) \setminus \{ 1_G \}$, then $h^{-1}\bdot S \in \mathcal B (G)$ has length $\mathsf D (G) + 1$ whence it has a factorization into at least two atoms, say
\[
  h^{-1}\bdot S = U_1 \bdot \ldots \bdot U_k \,, \quad \mbox{ where } \,\, k \ge 2, \,\, U_1, \cdots, U_k \in \mathcal A (G) \und h^{-1} \in \supp(U_1) \,.
\]
Since $h \neq 1_G$, we have that $|U_1| \ge 2$ and that $S$ has a proper product-one subsequence, a contradiction.
Thus we obtain that $\pi(S) = \{ 1_G \}$. Thus each two  elements in $\supp(S)$ commute whence $\la \supp(S) \ra = G$ is abelian, a contradiction.

\smallskip
(b) $\Rightarrow$ (a) If $G$ is abelian, then an atom of length $\mathsf D (G)$ does not have a proper non-trivial product-one subsequence.

\smallskip
2. If $G$ is abelian, then $|G'|=1$ and $G/G' \cong G$. Suppose that $G$ is non-abelian and consider a sequence
$T={g_1}G' \bdot \ldots \bdot {g_{\ell}}G' \in \mathcal F (G/G')$  with $\ell = \mathsf D (G)$ and $g_1, \ldots, g_{\ell} \in G$.
By 1., the sequence
\[
  S = g_1 \bdot \ldots \bdot g_{\ell} \in \mathcal F (G)
\]
has a proper product-one subsequence whence $T$ has  a proper product-one subsequence. Thus we obtain that
\[
  \mathsf D (G/G') = \mathsf d (G/G') + 1 \,\, \le \,\, |T | = \mathsf D (G) \,.
\]

\smallskip
3. If $S$ has the form given in (b), then clearly $\{2, \mathsf D (G)\} \subset \mathsf L (S)$. Conversely, suppose that
 $\{2, \mathsf D (G)\} \subset \mathsf L (S)$. Then there are $U_1, U_2, V_1, \ldots, V_{\mathsf D (G)} \in \mathcal A (G)$ such that
\[
  S = U_1 \bdot U_2 = V_1 \bdot \ldots \bdot V_{\mathsf D (G)} \,.
\]
If $1_G \t S$, then after renumbering if necessary it follows that $U_1 = 1_G = V_1$ whence $\mathsf D (G)=2$, a contradiction to $|G|\ge 3$.   Therefore,  $1_G \nmid S$ and hence $| V_k | \geq 2$ for all $k \in [1,\mathsf D (G)]$.
Then we obtain that
\[
  2\mathsf D (G) \leq | V_1\bdot \ldots \bdot V_{\mathsf D (G)} | = |S| = | U_1\bdot U_2 | \leq 2\mathsf D (G) \,,
\]
and it follows that $|S| = 2\mathsf D (G)$. Thus $|V_k| = 2$ for all $k \in [1,\mathsf D (G)]$,  $|U_1| = |U_2| = \mathsf D (G)$, and hence $U_2 = U^{-1}_1$.
\end{proof}

\smallskip
We need the following result for finite abelian groups (see \cite[Theorem 6.6.3]{Ge-HK06a}).

\smallskip
\begin{lemma} \label{4.3}~
For a finite abelian groups $G^*$ with $|G^*| \ge 3$ the following statements are equivalent{\rm \,:}
\begin{enumerate}
\item[(a)] Every $L \in \mathcal L (G^*)$ with $\{2, \mathsf D (G^*)\} \subset L$ satisfies \ $L = \{2, \mathsf D (G^*) \}$.

\smallskip
\item[(b)] $\{ 2, \mathsf  D (G^*) \} \in \mathcal L (G^*)$.

\smallskip
\item[(c)] $G^*$ is either cyclic or an elementary $2$-group.
\end{enumerate}
\end{lemma}

\smallskip
\begin{theorem} \label{4.4}~
Let $G$ and $G^*$ be finite groups and $n \in \N_{\ge 3}$ be odd.
\begin{enumerate}
\item If $|G|$ is odd, then $\{ 2, \mathsf D (G) \} \in \mathcal L (G)$ if and only if $G \cong C_{\mathsf D (G)}$.

\smallskip
\item If $|G|$ is odd and $G^{*}$ is  cyclic such that $\mathcal L (G^{*}) = \mathcal L (G)$, then $G^{*} \cong G$.

\smallskip
\item If $G=D_{2n}$, then $\{2, \mathsf D (G)\} \in \mathcal L (G)$ and there is an $L \in \mathcal L (G)$ with $\{2, \mathsf D (G)\} \subsetneq L$.

\smallskip
\item If $G^*$ is abelian, then $\mathcal L (G^*) \ne \mathcal L (D_{2n})$.
\end{enumerate}
\end{theorem}

\begin{proof}
1. By Lemma \ref{4.3} it suffices to show that $\{ 2, \mathsf D (G) \} \notin \mathcal L (G)$ for non-abelian group $G$ of odd order. Let $G$ be a non-abelian group of odd order, and assume to the contrary that $\{ 2, \mathsf D (G) \} \in \mathcal L (G)$. Then there is a $S \in \mathcal F (G)$ such that $\mathsf L (S) = \{ 2, \mathsf D (G) \}$. Then Lemma \ref{4.2}.3 implies that $S = U\bdot U^{-1}$ for some $U \in \mathcal A (G)$ of length $\mathsf D (G)$. By Lemma \ref{4.2}.1, $U$ has a non-trivial proper product-one subsequence, in particular an atom $V$.
If $|V| \ge 3$, then there exists an $k = \mathsf D (G) - |V| + 2 \in [3, \mathsf D (G)-1]$ such that $k \in \mathsf L (S)$, a contradiction. Thus $|V| = 2$, and it follows that any proper product-one subsequence dividing $U$ is a product of atoms of length $2$.

Since $|G|$ is odd, it follows that there is no element $g \in \supp(U)$ with $\ord(g) = 2$ such that $g\bdot g \t U$.
Then we set
\[
  U = g^{[r_1]}_1 \bdot (g^{-1}_1)^{[r'_1]} \bdot \ldots \bdot g^{[r_{\ell}]}_{\ell} \bdot (g^{-1}_{\ell})^{[r'_{\ell}]} \bdot T \,,
\]
where $r_1, r'_1, \ldots , r_{\ell}, r'_{\ell} \in \N, g_1, \ldots , g_{\ell}$ are distinct, and $T \in \mathcal F (G)$ is product-one free such that $g_i, g^{-1}_i \notin \supp(T)$ for each $i \in [1,\ell]$.
We may assume that $r_1, r'_1, \ldots , r_{\ell}, r'_{\ell}$ are maximal with respect to such expression.
Consider the sequence
\[
  W = g^{[r_1 + r'_1]}_1 \bdot \ldots \bdot g^{[r_{\ell} + r'_{\ell}]}_{\ell} \bdot T \quad \und \quad W^{-1} = (g^{-1}_1)^{[r_1 + r'_1]} \bdot \ldots \bdot (g^{-1}_{\ell})^{[r_{\ell} + r'_{\ell}]} \bdot T^{-1}
\]
of length $\mathsf D (G)$, and hence $U\bdot U^{-1} = W\bdot W^{-1}$.
Applying again Lemma \ref{4.2}.1, $W$ has a non-trivial proper product-one subsequence, in particular an atom $W'$.
As at the start of the proof, we obtain $|W'| = 2$, and it follows that there are $i, j \in [1,\ell]$ with $i \neq j$ such that $g_i = g^{-1}_j$, a contradiction to the maximality of $r_1, r'_1, \ldots , r_{\ell}, r'_{\ell}$.

\smallskip
2. Let $|G|$ be odd and  $G^{*}$ be cyclic with $\mathcal L (G^{*}) = \mathcal L (G)$. Then, by \cite[Proposition 5.6]{Oh18a},
\[
  \mathsf D (G^{*}) = \rho_2 (G^{*}) = \rho_2 (G) = \mathsf D (G) \,,
\]
whence 1. and Lemma \ref{4.3} imply the assertion.

\smallskip
3. We have $\mathsf D (G) = 2n$ by \cite[Theorem 1.1]{Ge-Gr13a}. We set $G_0 = \{ b, ab\}$ and assert that $\mathcal A (G_0) = \{ b^{[2]}, (ab)^{[2]},  S = b^{[n]}\bdot (ab)^{[n]}\}$. We leave this proof to the reader. Since $\mathcal B (G_0)$ has precisely three atoms, it follows immediately that
 $\mathsf L (S\bdot S^{-1}) = \mathsf L (S \bdot S) = \{2, \mathsf D (G)\}$. Furthermore,
\[
T = a^{[2n-2]} \bdot b^{[2]} \in \mathcal A (G) \ \text{(for details see the proof of \cite[Lemma 5.2]{Gr13b})} \,,
\]
and since
\[
T \bdot T^{-1} = a^{[n]} \bdot (a^{-1})^{[n]} \bdot \big( a \bdot a^{-1} \big)^{[n-2]} \bdot b^{[2]} \bdot b^{[2]} \,,
\]
we infer that $\{2, n+2, \mathsf D (G) \} \subset \mathsf L (T \bdot T^{-1})$.

\smallskip
4. Let $G^*$ be a finite abelian group. Lemma \ref{4.3} together with 3. implies that $\mathcal L (D_{2n}) \ne \mathcal L (G^*)$.
\end{proof}

\smallskip
\begin{lemma} \label{4.5}~
\begin{enumerate}
\item For every finite group $G$, we have
      \[
      \mathsf D (G) \le |G| \le 1 + \sum_{k=1}^{\mathsf d (G)} \binom{\mathsf d (G)}{k} k! \,.
      \]

\smallskip
\item For every $N \in \N$, there are only finitely many finite groups $G$ $($up to isomorphism$)$ such that $\mathsf D (G)=N$.

\smallskip
\item For every finite group $G$, there are only finitely many finite group $G^{*}$ $($up to isomorphism$)$ such that $\mathcal L (G^{*}) = \mathcal L (G)$.
\end{enumerate}
\end{lemma}

\begin{proof}
1. The left inequality follows from Lemma \ref{2.2}.1.  Consider a product-one free sequence $S \in \mathcal F (G)$ of length $|S| = \mathsf d (G)$.
Then $\Pi(S) = G \setminus \{1_G\}$ by Lemma \ref{2.2}.2 whence it follows  that
\[
   |G| \,\, = \,\, 1 + |\Pi(S)| = 1 + \Big| \bigcup_{1 \ne T\t S} \pi (T) \Big|  \le 1 + \sum_{k=1}^{\mathsf d (G)} \binom{\mathsf d (G)}{k} k! \,.
\]

\smallskip
2. It is well-known that for every $M \in \N$ there are only finitely many finite groups $G$ (up to isomorphism) with $|G| \le M$. Thus the assertion follows from 1.

\smallskip
3. Let $G$ be a finite group. If $G^{*}$ is any finite group with $\mathcal L (G^{*}) = \mathcal L (G)$, then by \cite[Proposition 5.6]{Oh18a},
\[
   \mathsf D (G^{*}) = \rho_2 (G^{*}) = \rho_2 (G) = \mathsf D (G) \,.
\]
Thus the assertion follows from 2.
\end{proof}

\smallskip
\begin{lemma} \label{4.6}~
\begin{enumerate}
\item If $G$ is a finite  group with $|G| \ge 32$, then either $G \cong C_2^5$ and $\mathsf d (G)=5$ or else $\mathsf d (G) \ge 6$.

\smallskip
\item Every group $G$ with $\mathsf D (G)=6$ is isomorphic to one of the following groups{\rm \,:}
      \[
        C_6, \quad C_2 \oplus C_2 \oplus C_4, \quad C_2^5, \quad D_6, \quad Q_8,  \quad D_8 \,.
      \]
\end{enumerate}
Moreover, all finite groups $G$ with $\mathsf D (G) \le 5$ are abelian.
\end{lemma}

\begin{proof}
Suppose that $G$ is abelian, say $G \cong C_{n_1} \oplus \ldots \oplus C_{n_r}$ with $1 < n_1 \mid \ldots \mid n_r$. We use frequently that $\sum_{i=1}^r (n_i-1) \le \mathsf d (G)=\mathsf D (G)-1$ and that equality holds for $p$-groups and in case $r \le 2$ (see \cite[Chapter 5]{Ge-HK06a}).

1. Let $G$ be a finite group with $|G| \ge 32$. If there is an element $g \in G$ with $\ord (g)\ge 7$, then $g^{[6]}$ is product-one free whence
\[
\mathsf d (G) \ge |g^{[6]}| = 6 \,.
\]
Suppose that $\ord (g') \le 6$ for all $g' \in G$ and that there is an element $g \in G$ with $\ord (g)=6$. Since $|G|\ge 32$, it follows that $\langle g \rangle \subsetneq G$ whence there is some $h \in G \setminus \langle g \rangle$. Then $g^{[5]} \bdot h$ is product-one free, and we assume from now on that all group elements have order at most five.

First we suppose that $G$ is abelian, say $G \cong C_{n_1} \oplus \ldots \oplus C_{n_r}$ with $1 < n_1 \t \ldots \t n_r$.
If $n_r=2$, then $\mathsf d (G)=r$ whence either $r=5$ and $G \cong C_2^5$ or else $\mathsf d (G)=r \ge 6$. If $n_r \in [3,5]$, then $\mathsf d (G) \ge \sum_{i=1}^r (n_i-1)$ quickly implies that assertion.

Suppose that $G$ is non-abelian. If all elements of $G$ would have order two, then $G$ would be abelian. Furthermore, since all groups of order $p^2$, for a prime $p$, are abelian, we suppose that $|G|$ is not the square of a prime.
Thus we may assume that all group elements have order at most five and $G$ has an element $g$ with $\ord (g) \in [3,5]$. Furthermore,  \cite[Proposition 3.9.1]{Cz-Do-Ge16a} shows us that for every normal subgroup $N \triangleleft G$ we have
\begin{equation} \label{normal-subgroup}
\mathsf d (N) + \mathsf d (G/N)\le \mathsf d (G) \,.
\end{equation}

\medskip
\noindent CASE 1{\rm \,:} \, $G$ is a $p$-group.

If $p=2$, then, by Sylow's Theorem, $G$ has a proper subgroup $H \subsetneq G$ with $|H|=16$. By the table given in \cite{Cz-Do-Sz17}, we obtain that $5 \le \mathsf d (H) < \mathsf d (G)$.
If $p=3$, then $G$ has a subgroup $H$ with $|H|=27$ and the table given in \cite{Cz-Do-Sz17} shows that $\mathsf d (H) \ge 6$. If $p=5$, then $G$ has a subgroup $H$ with $|H|=25$ whence $H$ is abelian and $\mathsf d (H)\ge 8$.

\medskip
\noindent CASE 2{\rm \,:} \, $G$ is not a $p$-group.

Then $|G| = 2^{i}3^{j}5^{k}$ is not a prime power, where $i \in [0, 4]$, $j \in [0, 2]$, and $k \in [0, 1]$. Suppose that $G$ is simple. Then \cite{Ca1977} implies that $G$ is isomorphic either to  $A_5$ or to $A_6$. Since $D_{10}$ is a proper subgroup of both $A_5$ and $A_6$,  \cite[Corollary 5.7]{Gr13b} implies that
\[
  5 = \mathsf d (D_{10}) < \mathsf d (G) \,.
\]
From now on we suppose that $G$ is not simple, and we denote by $Syl_p (G)$ a Sylow $p$-subgroup of $G$ for a prime number $p$.

\medskip
\noindent CASE 2.1{\rm \,:} \, $k=0 \und j = 1$.

Then  we only have the case $i=4$. Then, by Sylow's Theorem, we obtain that either $Syl_2 (G)$ or $Syl_3 (G)$ is normal subgroup of $G$. If $N = Syl_2 (G)$ is normal subgroup, then $|N| = 16$, and hence $4 \le \mathsf d (N)$ by the table given in \cite{Cz-Do-Sz17}. It follows that
\[
  6 \le \mathsf d (N) + \mathsf d (G/N) \le \mathsf d (G) \,.
\]
If $Syl_3 (G)$ is normal subgroup, then we obtain the same result by symmetry.

\medskip
\noindent CASE 2.2{\rm \,:} \, $k=0 \und j=2$.

If $i = 2$, then it is well known that either $Syl_2 (G)$ or $Syl_3 (G)$ is a normal subgroup of $G$ (see, \cite[Exercise 6.2.18]{Du-Fo04}).
By letting $N = Syl_2 (G)$ or $N = Syl_3 (G)$, we obtain in any cases
\[
  6 \le \mathsf d (N) + \mathsf d (G/N) \le \mathsf d (G) \,.
\]
Suppose that $i=3$ and $N$ is a non-trivial proper normal subgroup of $G$. Consider all possible value of the pair
\[
  \big( |N|, |G/N| \big) \quad \in \quad \big\{ (2,36), \,\, (4,18), \,\, (6,12), \,\, (8,9), \,\, (9,8), \,\, (12,6), \,\, (18,4), \,\, (36,2) \big\} \,.
\]
Applying the table given in \cite{Cz-Do-Sz17} with the case $i=2$, we obtain that
\[
  7 \le \mathsf d (N) + \mathsf d (G/N) \le \mathsf d (G) \,.
\]
For $i=4$, let $N$ be a non-trivial proper normal subgroup of $G$ and consider all possible value of the pair
\[
  \begin{aligned}
  \big( |N|, |G/N| \big) \quad \in  & \quad \big\{ (2,72), \,\, (3,48), \,\, (4,36), \,\, (6,24), \,\, (8,18), \,\, (9,16), \,\, (12,12),\\
                                    & \quad \ \  (16,9), \,\, (18,8), \,\, (24,6), \,\, (36,4), \,\, (48,3), \,\, (72,2) \big\} \,.
  \end{aligned}
\]
Applying the table given in \cite{Cz-Do-Sz17} with the case $i = 2,3$ and CASE 4.1, we obtain that
\[
  8 \le \mathsf d (N) + \mathsf d (G/N) \le \mathsf d (G) \,.
\]

\medskip
\noindent CASE 2.3{\rm \,:} \, $k=1 \und j=0$.

By the same line of CASE 2.2, it is enough to verify the case $i=3$ and hence we assume that $|G| = 40$.
Then, by Sylow's Theorem, $N = Syl_5 (G)$ is normal subgroup of $G$. Hence $G/N$ is isomorphic to one of the following groups{\rm \,:}
\[
  C^{3}_2, \quad C_2 \oplus C_4, \quad C_8, \quad D_8, \quad Q_8 \,.
\]
Thus we obtain
\[
  7 \le \mathsf d (N) + \mathsf d (G/N) \le \mathsf d (G) \,.
\]

\medskip
\noindent CASE 2.4{\rm \,:} \, $k=1 \und j=1$.

By the same line of CASE 2.2, it is enough to verify the case $i=2$ and hence we assume that $|G| = 60$. Since $G$ is not simple group, we obtain that $N = Syl_5 (G)$ is normal subgroup of $G$ by \cite[Proposition 21]{Du-Fo04}. Hence $G/N$ is isomorphic to one of the following groups{\rm \,:}
\[
  C_2 \oplus C_6, \quad C_{12}, \quad A_4, \quad D_{12}, \quad Dic_{12} \,.
\]
Thus we obtain
\[
  8 \le \mathsf d (N) + \mathsf d (G/N) \le \mathsf d (G) \,.
\]

\medskip
\noindent CASE 2.5{\rm \,:} \, $k=1 \und j=2$.

By the same line of CASE 2.2, it is enough to verify the case $i=0$ and hence we assume that $|G| = 45$.
Then, by Sylow's Theorem, $N = Syl_3 (G)$ is normal subgroup of $G$. Hence $G/N$ is isomorphic to $C_5$. Thus we obtain
\[
  8 = \mathsf d (N) + \mathsf d (G/N) \le \mathsf d (G) \,.
\]

\medskip
2. Let $G$ be a finite  group with $\mathsf D (G)=6$. If $G$ is abelian,  say $G \cong C_{n_1} \oplus \ldots \oplus C_{n_r}$ with $1 < n_1 \mid \ldots \mid n_r$, then
\[
1 + \sum_{i=1}^r (n_i-1) \le \mathsf D (G) = 6 \,,
\]
implies  that $G$ is isomorphic to one of groups  in the list.
Suppose that $G$ is non-abelian. Then 1. implies that   $|G| \le 32$. Now the table given in \cite{Cz-Do-Sz17} shows that $G$ is isomorphic either to $D_6$, or to $Q_8$, or to $D_8$.
\end{proof}

\smallskip
It is easy to write down explicitly the system $\mathcal L (G)$  for groups $G$ with $\mathsf D (G) \le 4$ (\cite[Section 7.3]{Ge-HK06a}). However, it turned out that the explicit description of $\mathcal L (G)$ for groups with $\mathsf D (G)=5$ is extremely complex (\cite[Section 4]{Ge-Sc-Zh17b}), to the extent that no system $\mathcal L (G)$ has been  written down completely. Nevertheless,  we can show that for a group $G$ with Davenport constant $\mathsf D (G)=6$ its system of sets of lengths is characteristic.

\smallskip
\begin{theorem} \label{4.7}~
Let $G$ be a finite group with $\mathsf D (G) = 6$. If $G^{*}$ is a finite group with $\mathcal L (G^{*}) = \mathcal L (G)$, then $G^{*} \cong G$.
\end{theorem}

\begin{proof}
Let $G^{*}$ be a finite group with $\mathcal L (G^{*}) = \mathcal L (G)$. By \cite[Proposition 5.6]{Oh18a}, we infer that
\[
  \mathsf D (G^{*})= \rho_2 (G^{*}) = \rho_2 (G) = \mathsf D (G) = 6 \,.
\]
Lemma \ref{4.6} provides all groups having Davenport constant six.  Thus it remains to show that  the systems of sets of lengths of any two of them are distinct.
If $G$ and $G^*$ are both abelian, then $\mathcal L (G) \ne \mathcal L (G^*)$ by \cite[Theorem 7.3.3]{Ge-HK06a}. Thus it remains to consider the case where $G \in \{D_6, Q_8, D_8\}$. 

In the following three cases, we list all minimal product-one sequences of certain length. This can be done by straightforward but very tedious case distinctions or by computer (indeed, we rechecked the given lists with the help of Mathematica).

\smallskip
\noindent CASE 1{\rm \,:} \, $G=D_6$.

By Theorem \ref{4.4}.4, $G^*$ cannot be abelian. Hence it remains to verify that $\mathcal L (D_6) \ne \mathcal L (Q_8)$ and $\mathcal L (D_6) \ne \mathcal L (D_8)$. Observe that
\begin{itemize}
\item $\mathcal B (D_8)$ has precisely four atoms of length $6$, namely
      \[
        a^{[4]}\bdot b \bdot (a^{2}b), \quad (a^{3})^{[4]}\bdot b\bdot (a^{2}b), \quad a^{[4]}\bdot ab\bdot (a^{3}b), \quad (a^{3})^{[4]}\bdot ab\bdot (a^{3}b) \,.
      \]

\smallskip
\item The atoms of $\mathcal B (Q_8)$ of length $6$ have the following form{\rm \,:}
      \[
        g_{1}^{[4]}\bdot g_{2}^{[2]}\,, \quad \mbox{ where } \,\, g_1, g_2 \in \{ I, J, K ,-I, -J, -K \} \,\, \mbox{ with } \,\, g_2 \neq \pm g_1 \,.
      \]
\end{itemize}
It follows that $\{2,6\} \notin \mathcal L (D_8)$ and $\{2,6\} \notin \mathcal L (Q_8)$. Thus $\mathcal L (Q_8) \ne \mathcal L (D_6)$ and $\mathcal L (D_8) \ne \mathcal L (D_6)$ by Theorem \ref{4.4}.3.

\smallskip
\noindent CASE 2{\rm \,:} \, $G=Q_8$.

Since $\{2,6\} \notin \mathcal L (Q_8)$, Lemma \ref{4.3} implies that it suffices to show that $\mathcal L (C^{2}_{2} \oplus C_4) \neq \mathcal L (Q_8)$ and $\mathcal L (D_8) \neq \mathcal L (Q_8)$.
To do so we recall that  $\{ 2, 5 \} \in \mathcal L (C^{2}_{2} \oplus C_4)$ by \cite[Proposition 3.8]{Ge-Zh15b}, and for
$S = a^{2} \bdot b\bdot b\bdot ab \bdot ab \in \mathcal B (D_8)$, we have that $S \in \mathcal A (D_8)$ and 
\[
 \{2, 5\} = \mathsf L (S\bdot S^{-1}) \in \mathcal L (D_8) \,.
\]
Therefore it is sufficient to verify  that $\{ 2, 5 \} \notin \mathcal L (Q_8)$. Assume to the contrary that there are $U, V \in \mathcal A (Q_8)$ such that $\{2, 5\} = \mathsf L (U\bdot V)$. Observe that any atom of length $5$ has one of the following three forms{\rm \,:}
\begin{itemize}
\item $g_1^{[3]}\bdot g_2\bdot g_3$, \, where $g_1, g_2, g_3 \in \{ I, J, K ,-I, -J, -K \}$ with $g_2 \neq g_3$ and $g_2, g_3 \neq \pm g_1$.

\smallskip
\item $g_1^{[2]}\bdot g_2^{[2]}\bdot (-E)$, \, where $g_1, g_2 \in \{ I, J, K ,-I, -J, -K \}$ with $g_2 \neq \pm g_1$.

\smallskip
\item $g_1\bdot (-g_1)\bdot g_2\bdot (-g_2)\bdot (-E)$, \, where $g_1, g_2 \in \{ I, J, K \}$ with $g_1 \neq g_2$.
\end{itemize}

\smallskip
\noindent CASE 2.1{\rm \,:} \, $|U| = |V| = 5$.

Then $V = U^{-1}$, and hence we obtain that $4 \in \mathsf L (U\bdot U^{-1})$, a contradiction.

\smallskip
\noindent CASE 2.2{\rm \,:} \, $|U| = 6$ or $|V| = 6$.

Without loss of generality, we may assume that $|U| = 6$ and we set $U = g^{[4]}_1 \bdot g^{[2]}_2$ for some $g_1, g_2 \in \{ I, J, K ,-I, -J, -K \}$ with $g_2 \neq \pm g_1$.
If $|V| = 5$, then $U\bdot V$ has a factorization of product of four atoms of length $2$ and one atom of length $3$. It follows that
\[
  V = (-g_1)^{[3]}\bdot (-g_2)\bdot g_3 \quad \mbox{ or } \quad V = (-g_1)^{[2]} \bdot (-g_2)^{[2]} \bdot (-E) \,,
\]
where $g_3 \in \{ K, -K \}$. Hence we obtain that
\[
  U\bdot V = \Big(g_1 \bdot (-g_1)\Big)^{[2]} \bdot \Big( (g_1)^{[2]}\bdot (g_2)^{[2]} \Big) \bdot W \,,
\]
where $W = (-g_1)\bdot (-g_2) \bdot g_3$ or $W = (-g_2)^{[2]}\bdot (-E)$, and thus $4 \in \mathsf L (U\bdot V)$, a contradiction.
If $|V| = 6$, then
\[
  V = (-g_2)^{[4]}\bdot (-g_1)^{[2]} \quad \mbox{ or } \quad V = (-g_1)^{[4]}\bdot g^{[2]}_3 \,,
\]
where $g_3 \in \{ I, J, K ,-I, -J, -K \} \setminus \{ g_1, -g_1, -g_2 \}$. Hence we obtain that
\[
  U\bdot V = \Big(g_1 \bdot (-g_1)\Big)^{[2]} \bdot \Big( (g_1)^{[2]}\bdot (g_2)^{[2]} \Big) \bdot W \,,
\]
where $W = (-g_2)^{[4]}$ or $W = (-g_1)^{[2]}\bdot g^{[2]}_3$, and thus $4 \in \mathsf L (U\bdot V)$, a contradiction.

\smallskip
\noindent CASE 3{\rm \,:} \, $G=D_8$.

Since $\{2,6\} \notin \mathcal L (D_8)$, Lemma \ref{4.3} implies that it remains to show that $\mathcal L (C^{2}_{2} \oplus C_4) \neq \mathcal L (D_8)$.
We have that $S = a^{[4]} \bdot b\bdot (a^{2}b) \in \mathcal A (D_8)$ and
\[
  S\bdot S^{-1} = \Big(a^{[4]}\Big) \bdot \Big( (a^{3})^{[2]} \bdot b^{[2]} \Big) \bdot \Big( (a^{3})^{[2]} \bdot (a^{2}b)^{[2]} \Big) \,,
\]
whence $\{2, 3, 6 \} \subset \mathsf L (S\bdot S^{-1})$. On the other hand, by \cite[Proposition 4.14]{Sc16a}, there is no $L \in \mathcal L (C^{2}_{2} \oplus C_4)$ such that $\{2, 3, 6 \} \subset  L$. Thus $\mathcal L (C^{2}_{2} \oplus C_4) \neq \mathcal L (D_8)$.
\end{proof}

\providecommand{\bysame}{\leavevmode\hbox to3em{\hrulefill}\thinspace}
\providecommand{\MR}{\relax\ifhmode\unskip\space\fi MR }
\providecommand{\MRhref}[2]{%
  \href{http://www.ams.org/mathscinet-getitem?mr=#1}{#2}
}
\providecommand{\href}[2]{#2}

\end{document}